\input amssymb.sty
\documentclass{amsproc}
\usepackage{graphicx}
\usepackage{color}
\usepackage[english]{babel}
\usepackage[T1]{fontenc}
\usepackage{amsmath}
\usepackage{amscd}
\usepackage{amstext}
\usepackage{amsbsy}
\usepackage{amsopn}
\usepackage{amsthm}
\usepackage{amsxtra}
\usepackage{upref}
\usepackage{amsfonts}
\usepackage{amssymb}
\usepackage{mathrsfs}
\usepackage{euscript}
\usepackage{array}
\usepackage{stmaryrd}
\usepackage{verbatim}
\newtheorem{theorem}{Theorem}[section]
\newtheorem{lemma}[theorem]{Lemma}
\newtheorem{corollary}[theorem]{Corollary}
\newtheorem{proposition}[theorem]{Proposition}
\theoremstyle{definition}
\newtheorem{definition}[theorem]{Definition}
\newtheorem{definition-proposition}[theorem]{Definition-Proposition}

\theoremstyle{remark}
\newtheorem{remark}[theorem]{Remark}

\newcounter{Step}
\numberwithin{equation}{section}

\def\R{\mbox{I\hspace{-.15em}R} }

\def\C{\mathbf{C}}

\def\N{\mbox{I\hspace{-.15em}N} }

\newcommand{\tuchi}[0]{\tilde{\underline{\chi}}}
\newcommand{\tchi}[0]{\tilde{\chi}}
\newcommand{\tochi}[0]{\tilde{\chi}^{\omega}}
\newcommand{\totchi}[0]{\tilde{\chi}^{\omega,\theta}}
\newcommand{\tachi}[0]{\tilde{\chi}^{a}}
\newcommand{\red}{}
\newcommand{\gre}{}
\newcommand{\gree}{}
\newcommand{\bl}{}

\begin{document}

\title{Concavification of free entropy }

\author{Philippe Biane}
\address{\gre CNRS, IGM,
    Universit\'e Paris-Est, Champs-sur-Marne, FRANCE
}
\email{\gre biane@univ-mlv.fr}
\thanks{\gre Research of P. Biane partially supported by ANR grant GrandMa 08-BLAN-0311-01}
%\thanks{$^{(2)}$Authors want to thank }

\author{Yoann Dabrowski}
\address{\gre Universit\'{e} de Lyon\\ 
Universit\'{e} Lyon 1\\
Institut Camille Jordan UMR 5208\\
43 blvd. du 11 novembre 1918\\
F-69622 Villeurbanne cedex\\
France
}
\email{\gre dabrowski@math.univ-lyon1.fr}
%    General info
\date{}

\begin{abstract}
We introduce a modification of 
 Voiculescu's free entropy which coincides with  the $\liminf$ variant of
Voiculescu's  free entropy  on  extremal states, but is a concave
 upper semi-continuous function on the trace state space. 
{\gre We also extend the orbital free entropy of \cite{HiaiMUeda} to non-hyperfinite multivariable{\bl s} and prove freeness in case of additivity of Voiculescu's entropy (or vanishing of our extended orbital entropy).}

\end{abstract}

\maketitle
\section{Introduction}
Voiculescu has introduced a free entropy quantity, for tracial states on a von Neumann
algebra generated by $n$ self-adjoint elements, which
  has been very useful for the solution of many long
standing open problems in von Neumann algebra theory. It turns out that
free entropy
satisfies an unusual property for an entropy quantity which is a ``degenerate convexity"
property, i.e. the entropy of  any nonextremal state is $-\infty$, which is in
sharp contrast with the usual concavity and upper semi-continuity property of
classical entropy. Recently Hiai \cite{Hi} defined a free analogue of pressure and
considered its
Legendre transform. He obtained a quantity which is concave and upper
semi-continuous, and majorizes Voiculescu's free entropy. It is not clear
whether this quantity coincides with Voiculescu's free entropy on extremal
states. In this paper we introduce a modified definition, through random matrix
approximations, which yields a
quantity which is both concave upper semi-continuous, and coincides with the
$\liminf$ variant of
Voiculescu's  free entropy  on  extremal states. 
Our main argument is the simple observation that a probability
measure on a 
compact convex set,
whose barycenter is close to an extremal point, has most of its mass
concentrated near this point (see Lemma \ref{concentration} below). This is obvious in finite dimension, but requires further clarification  in infinite dimension. In this paper we rely on the fact that the convex set we consider is a {\gre Poulsen simplex}. 

We use an analogous idea to generalize the definition of free orbital entropy, due to
Hiai, Miyamoto and Ueda \cite{HiaiMUeda}. In this paper, the authors introduced, via  a microstates
approach, an entropy quantity
$\chi_{orb}(\mathbf{X_1},\ldots,\mathbf{X_n})$, where each $\mathbf{X_i}$ is a finite set of noncommutative random variables generating a hyperfinite algebra.
They used this quantity to  generalize Voiculescu's additivity result (\cite{Voi4}),  namely~: for  noncommutative random variables
$X_1,\ldots,X_n$, if $$\chi(X_1,\ldots,X_n)=\chi(X_1)+\ldots+\chi(X_n)$$ and these quantities are finite, then
the $X_i$ are free. More generally, they showed that $\chi_{orb}(\mathbf{X_1},\ldots,\mathbf{X_n})=0$ is equivalent to freeness in the hyperfinite  context above even though the finiteness of entropy fails in general in this case. They recover the previous result since they also show~: $$\chi(X_1,\ldots,X_n)=\chi_{orb}(X_1,\ldots,X_n)+\chi(X_1)+\ldots+\chi(X_n),$$ in case these quantities are finite. 

 In section 7,  we introduce a definition of $\chi_{orb}(\mathbf{X_1},\ldots,\mathbf{X_n})$, for arbitrary finite sets
$\mathbf{X_i}$ of noncommutative random variables, obtained by replacing microstates by probability measures. We show that many of the arguments of \cite{HiaiMUeda} have analogues  in this setting, and we obtain the full generalization of the additivity result when random variables $X_i$ are replaced by arbitrary finite sets 
$\mathbf{X_i}$.

This paper is organized as follows.
We start by recalling some well known facts on trace states and on Legendre
transform and classical entropy (including Csiszar's projections result) in section 2 and 3. Then we prove the main result about concavification in section 4 and 6, after a few preliminaries about Poulsen simplices in section 5. In section 7 we extend the definition of orbital entropy,  and prove freeness in case of additivity of Voiculescu's entropy, in Corollary \ref{additivity}. Finally,  after a few more preliminaries in section 8, section 9 is devoted to some further variants and extensions of our definitions, which might prove useful for future applications.

 \bigskip\textbf{Acknowledgments :} The authors want to thank respectively U. Haagerup and D. Shlyakhtenko for fruitful discussions. The authors are grateful to the Erwin Schr\"{o}dinger Institute where part of this work has been completed. They also thank the organizers of the workshop on "Random Matrix, Operator Algebra, and Mathematical Physics Aspects" in the semester on Bialgebras in free Probability having taken place there. Finally we would like to thank an anonymous referee for some very helpful comments and suggestions.

     \section{The set of trace states}
     Let ${\bf C}\langle
     X_1,\ldots,X_n\rangle$ be the free $*$-algebra with unit generated by
     $n{ \geq 1}$ self-adjoint elements $X_1,\ldots,X_n$, which we identify with the    space of noncommutative polynomials in the indeterminates $X_1,\ldots,X_n$.
     We consider the set { $\mathcal S_c^n$} of trace states on ${\bf C}\langle
     X_1,\ldots,X_n\rangle$. This set  consists in all positive, tracial 
      $*$-linear maps $\tau:{\bf C}\langle
     X_1,\ldots,X_n\rangle\to {\bf C}$ such that $\tau(1)=1$ and, 
for any $P\in {\bf C}\langle     X_1,\ldots,X_n\rangle $  there exists
      some constant
     $R_P>0$ such that {\red \begin{equation}
     \label{norm}\tau((P^{*}P)^k)\leq R_P^{2k}\qquad\text{
      for }k\geq 0
     \end{equation}}
     Let us denote by $\mathcal S_R^n$ the set of all trace states
      such that $\max(R_{X_1},...,R_{X_n})\leq R$.
      
      { Especially, for $R\geq T$ we have $\mathcal S_R^n\supset \mathcal S_T^n$. Moreover, $\mathcal S_c^n=\cup_{R\geq 0}S_R^n$. Finally for $\tau\in \mathcal S_c^n$, we define $\mathcal{R}(\tau)=\inf\{ R, \tau \in S_R^n\}$ so that obviously $\tau \in S_{\mathcal{R}(\tau)}^n$.}
      
     The set { $\mathcal S_R^n$} can be identified with the set of trace states on the free product
     $C^*$-algebra $*_{i=1}^nC([-R,R])$, cf \cite{Hi}. It is a compact convex 
     set for the 
     weak$^*$ topology. By the reduction theory for von Neumann algebras, it is a Choquet simplex,
     and its extreme points (for $n\geq 2)$ are the factor states 
       \cite{Tak}. { Note that, as a consequence, an extreme point in { $\mathcal S_R^n$} is still an extreme point in { $\mathcal S_T^n$} for $T\geq R$.} Moreover, the second author proved in \cite[Corollary 5]{D} that, for  ${ n > 1},$ { $\mathcal S_R^n$} is a Poulsen Simplex, i.e. the unique metrizable Choquet simplex with a dense set of extreme points (cf. \cite{L}). 
If $\mathcal A$ is a von Neumann algebra equipped with a tracial state $\varphi$, and
           $(X_1,\ldots,X_n)\in \mathcal A$  an $n$-tuple such that
$\sup_i\Vert X_i\Vert\leq R$, one defines a state $\tau_{X_1,\ldots,X_n}\in
      \mathcal S^n_R$ by the formula
     $$ \tau_{X_1,\ldots,X_n}(P)=\varphi(P(X_1,\ldots,X_n)).$$
In particular, if $\mathcal A=M_N({\bf C})$ and $\varphi=\frac{1}{N}Tr$ the normalized trace,
we denote by $H_N^R$ the set of hermitian matrices of size $N$, whose
     operator norm is less than $R$, then
      an $n$-tuple
      $(M_1,\ldots,M_n)\in (H_N^R)^n$ defines a state $\tau_{M_1,\ldots,M_n}\in
       \mathcal S^n_R$, by
      $$ \tau_{M_1,\ldots,M_n}(P)=\frac{1}{N}Tr(P(M_1,\ldots,M_n)).$$
      Similarly, a probability measure $\mu$ on $(H_N^R)^n$  (always assumed Borel) defines 
     a random state in $\mathcal S^n_R$, whose barycenter $\tau_{\mu}$, defined
     by
    $$\tau_{\mu}(P)=\int_{(H_N^R)^n}\frac{1}{N}Tr(P(M_1,\ldots,M_n))d\mu(M_1,\ldots,M_n),$$
 is again an element of
      $\mathcal S^n_R$.
     
      For $\tau\in \mathcal S^n_R$, let $V_{\epsilon,K}(\tau)$ be the set of
      states $\sigma\in \mathcal S^n_R$ such that,
      for all monomials $m$ of degree less than $K$, {\gree we have}~:
      $$|\tau\left(m(X_1,\ldots, X_n)\right)-\sigma\left(m(X_1,\ldots,X_n)\right)|<
     \epsilon.$$
     The sets $(V_{\epsilon,K}(\tau);\epsilon,K>0 )$ form a basis of 
     neighbourhoods of $\tau$ in the
     weak$^*$ topology.
      \section{Classical entropy, its Legendre transform {\gre and Csiszar's projection}}
      Recall that the entropy of a probability measure $\mu$ on $\bf{R}^p$ is the
      quantity
      $$\text{Ent}(\mu)=\left\lbrace\begin{array}{l}-\int_{\bf{R}^p} f(x)\log f(x)dx\quad
      \text{if}\ \mu(dx)=f(x)dx,{ \log(f)\in L^1(\mu)}\\ \\
      -\infty \quad\text{ otherwise}\end{array}\right .$$
      The entropy is a concave upper semi-continuous function of $\mu$.
      
      Moreover, there is also a well known notion of relative entropy of two probability measures (also called Kullback-Leibler divergence, cf. \cite{K}).
      $$\text{Ent}(\mu|\nu)=\left\lbrace\begin{array}{l}-\int_{\bf{R}^p} f(x)\log f(x)d\nu(x)\quad
      \text{if}\ \mu(dx)=f(x)d\nu(x),%{ \log(f)\in L^1(\mu)}
      \\ \\
      -\infty \quad\text{if $\mu$ is not absolutely
      continuous with respect to $\nu$
      }\end{array}\right .$$
Note that, by Jensen inequality, $\text{Ent}(\mu|\nu)\leq 0$. The relative entropy
  satisfies the following key property: For any  measurable map 
$T$, { if $T_*\mu$ is the pushforward measure of $\mu$, { we have}  (cf. \cite[Chap 2 Th 4.1]{K}):}
\begin{equation}\label{push}
 \text{Ent}(T_{*}\mu|T_{*}\nu)\geq \text{Ent}(\mu|\nu).
\end{equation} 
             If $E\subset \bf {R}^p$ is a subset with positive Lebesgue
      measure, and $\mu$ is the normalized Lebesgue measure on $E$, then
      $$\text{Ent}(\mu)=\log(\text{Leb}(E)) .$$
Actually this is the maximum value of $\text{Ent}$ on the set of
 all probability
      measures supported by $E$.
Analogously, if $\mu$ is the restriction of $\nu$ to $E$, renormalized into a probability measure, then 
$$\text{Ent}(\mu|\nu)=\log(\nu(E))).$$
and again this is the maximum value of $\text{Ent}(.|\nu)$ on the set of
 all probability
      measures supported by $E$.
From this we deduce the following estimates.

\begin{lemma}
Let $\mu$ %and $\nu$
 be supported by $E$ and $F\subset E$ a measurable subset, then

      \begin{equation}
      \begin{array}{rcl}\label{ent}
      \text{Ent}(\mu)&\leq &\mu(F)\log\text{Leb(F)}+\mu(E\setminus F)
      \log\text{Leb}(E\setminus F)\\&&\qquad-\mu(F)\log\mu(F)-(1-\mu(F))
      \log(1- \mu( F))\end{array}
      \end{equation}
and
      \begin{equation}
      \begin{array}{rcl}\label{entnu}
      \ \text{Ent}(\mu|\nu)&\leq & \mu(F)\log\nu(F)+\mu(E\setminus F)
      \log\nu(E\setminus F)\\&&  \qquad-\mu(F)\log\mu(F)-(1-\mu(F))
      \log(1- \mu( F)).
           \end{array}
           %{\gre \ \ \ (resp.\ \text{Ent}(\mu|\nu)
      \end{equation}
 \end{lemma}     

\begin{proof}
\begin{equation*}
      \begin{array}{rcl}
      \text{Ent}(\mu)&=&-\int_F f(x)\log f(x)dx-\int_{E\setminus F}f(x)\log
      f(x)dx\\
      &=&-\mu(F)\int_F \frac{f(x)}{\mu(F)}\log \frac{f(x)}{\mu(F)}
      dx-\mu(E\setminus F)\int_{E\setminus F}\frac{f(x)}{\mu(E\setminus F)}\log
      \frac{f(x)}{\mu(E\setminus
      F)}dx\\&&\qquad-\mu(F)\log\mu(F)-\mu(E\setminus F)\log \mu(E\setminus F)\\
      &\leq &\mu(F)\log\text{Leb(F)}+\mu(E\setminus F)
      \log\text{Leb}(E\setminus F)\\&&\qquad-\mu(F)\log\mu(F)-(1-\mu(F))
      \log(1- \mu( F)).\end{array}
      \end{equation*}
The proof of the other inequality is similar (cf. {\bl \cite[Chap 2 Cor 3.2]{K}}).
\end{proof}
      We shall need another characterization of entropy, through its Legendre
      transform. Indeed {\gree we have}, for any  probability measure $\mu$ supported by a set $E$, of finite Lebesgue measure,
      $$\text{Ent}(\mu)=\inf_{\phi\in C_b(E)}
       \left(\log\left(\int_E\exp \phi(x)dx\right)-\int_E\phi(x)\mu(dx)\right).$$
{ where $C_b(E)$ is the space of bounded, real valued continuous  functions on $E$}. Likewise (see e.g. \cite{DZ} section 6.2)    for any  probability measures $\mu,\nu$ supported on $E$,
      \begin{equation}\label{dualEnt}\text{Ent}(\mu|\nu)=\inf_{\phi\in C_b(E)}
       \left(\log\left(\int_E\exp \phi(x)d\nu(x)\right)-\int_E\phi(x)\mu(dx)\right).\end{equation}   
       
       It follows that if $f_1,\dots, f_p$ are { real valued} bounded measurable
        functions on $E$, then {\gree we have}
      \begin{equation}\label{maxent}
     \begin{array}{c} \inf_{\lambda\in \mathbb R^p}
     \left(\log\int_Ee^{\sum_i\lambda_if_i(x)}dx
      -\sum_ia_i\lambda_i\right)=\\
      \sup\left\{\text{Ent}(\mu)\,|\, \mu \text{ supported on }E;\,\int
      f_i(x)\mu(dx)=a_i,i=1,\ldots,p\right\}
      \end{array}
      \end{equation}
      where the $\sup$ is defined as $-\infty$ if there is no such probability
      measure.
      
    We will apply these considerations to the case where
    the set $E$ is a product of balls $H_N^R$, i.e. balls of radius $R$
    for the operator norm in the space of $N\times N$ hermitian matrices, with
     Lebesgue measure, and
    the functions $f_1,\ldots, f_p$ are traces of selfadjoint polynomials in
    noncommuting indeterminates,
     of the form $$f(M_1,\ldots, M_n)
    =NTr(P(M_1,\ldots,M_n)).$$
    Let us define
    $$I_N(P)=\int_{(H_N^R)^n}e^{-NTr(P(M_1,\ldots,M_n))}dM_1\ldots dM_n,$$ for $P$ 
     a self-adjoint  element of $\bf{C}\langle X_1,\ldots,X_n\rangle$.
    \begin{definition} For {\red  $\tau\in\mathcal S_R^{n}$, we define
    $\rho_{N,K}
   (\tau)$ as
   the maximum of the entropy of (Borel) probability  measures $\mu$ on $(H_N^R)^{n}$ 
   whose barycenter  coincides with $\tau$ on monomials of degree less than $K$,}
   and $\rho_{N,K}
   (\tau)=-\infty$ if there is no such measure. { Equivalently, if $P[(H_N^R)^{n}]$ is the the above set of Borel probability measures, we have~: $$\rho_{N,K}
   (\tau)=\sup_{{\mu \in P[(H_N^R)^{n}]\atop \,\tau_\mu\in
\cap_{\epsilon>0} V_{\epsilon,K}(\tau)} }Ent(\mu).$$
   
   }
   \end{definition}

   {\gree We have}, by (\ref{maxent}) :
     \begin{equation}\label{Legendre}\rho_{N,K}
   (\sigma)=
   \inf_{P\in\mathbf{C}\langle
     X_1,\ldots,X_n\rangle\atop  P=P^*,\text{deg}(P)\leq
     K}\left(\log I_N(P)+N^2\sigma(P)\right),
\end{equation}
   which is therefore  a concave upper semi-continuous
    function of $\sigma$.
    
    {\gre
     Even though we won't need it before section 9, it may be entlightening to use the language of Csiszar's I-projections (cf. \cite{Cis}, see also \cite[Chapter 10]{N} for an exposition). Let us recall the basics.
Let $\mathcal{E}$ be a closed convex set of probability distributions then, by the strict concavity of relative entropy, there exists a unique probability measure realizing  $\sup_{\mu\in \mathcal{E}}\text{Ent}(\mu|\nu)$. This probability distribution, denoted $C$, is called 
Csiszar's I-projection  of the probability distribution $\nu$ on the convex set $\mathcal{E}$.  Csiszar \cite{Cis} first proved its existence when $\mathcal{E}$ is variation closed and contains a $\mu$ with $\text{Ent}(\mu|\nu)>-\infty$. Moreover $C$ is characterized by : $$\text{Ent}(\mu|\nu)\leq\text{Ent}(\mu|C)+\text{Ent}(C|\nu),$$ for every $\mu\in \mathcal{E}$.
     We can infer from this that   $\rho_{N,K}
   (\tau)$, { if finite}, is the entropy  of Csiszar's I-projection { $C_{N,0,K}(\tau)$} of normalized Lebesgue measure { (on $(H_N^R)^n$)} on the set  of measures { whose mean agrees with $\tau$ on monomials of order less than $K$}. 
  { It is a well known result about exponential families (see e.g. \cite[Theorem 3.1]{Cis} or \cite[Theorem 10.2]{N}) that $C_{N,0,K}(\tau)$ has a density with respect to normalized Lebesgue measure on $(H_N^R)^n$ of the form $\frac{1}{Z}e^{-Tr(V(X))}$ for a non commutative polynomial $V$ of degree less than $K$. Especially, $\rho_{N,K}
   (\tau)$ is the entropy of a well-studied unitary invariant random matrix model. }   }

     \section{Voiculescu's free entropy and its modification}\label{VFE}
     Let $\tau\in\mathcal S_R^n$, let $\epsilon>0$ be a real number and
     $K,N$ be positive integers. We denote by
     $\Gamma_R(\tau,\epsilon,K,N)$ the set of $n$-tuples of hermitian matrices
     $M_1,\ldots,M_n\in H_N^R$ such that
     for all monomials $m(X_1,\ldots,X_n)=X_{i_1}\ldots X_{i_k}$ of degree less than $K$ {\gree we have}~:
     $$|\tau\left(m(X_1,\ldots,
     X_n)\right)-\frac{1}{N}Tr\left(m(M_1,\ldots,M_n)\right)|<
     \epsilon$$
     Equivalently $\Gamma_R(\tau,\epsilon,K,N)$ is the set 
     of $n$-tuples of hermitian matrices
     $M_1,\ldots,M_n\in H_N^R$ whose associated state $\tau_{M_1,\ldots, M_n}$
     is in $V_{\epsilon,K}(\tau)$.

\begin{definition}\cite{V1} Define { for $\tau\in\mathcal S_R^n$ :}
          $$\chi_R(\tau)=\lim_{K\to\infty,\epsilon\to 0}
     \limsup_{N\to\infty}\left(\frac{1}{N^2}\log\left(
     \text{Leb}(\Gamma_R(\tau,\epsilon,K,N)\right))+\frac{n}{2}\log N\right).$$
The free entropy of a tracial state $\tau{ \in\mathcal S_c^n}$ is~:
$$\chi(\tau)=\sup_{R\geq { \mathcal{R}(\tau)}}\chi_R(\tau){.}$$
\end{definition}     
 It is known that, if $\tau$ is not an
     extreme point of $\mathcal S_R^n$, then $\chi(\tau)=-\infty$, cf \cite{Voi3}.
     Furthermore, if $\tau$ is considered as a state in $\mathcal S^n_{R'}$ for
     some $R'>R\ { >\mathcal{R}(\tau)}$ then 
     $\chi_{R'}(\tau)=\chi_R(\tau)$.
        Since it is not known whether the $\limsup$ in the definition is a limit, it has been useful to define~:
     $$\underline{\chi}_R(\tau)=\lim_{K\to\infty,\epsilon\to 0}
     \liminf_{N\to\infty}\left(\frac{1}{N^2}\log\left(
     \text{Leb}(\Gamma_R(\tau,\epsilon,K,N)\right))+\frac{n}{2}\log N\right)$$
and, for  a nontrivial ultrafilter $\omega$ on $\bf N$ :
     $$\chi_R^{\omega}(\tau)=\lim_{K\to\infty,\epsilon\to 0}
     \lim_{N\to\omega}\left(\frac{1}{N^2}\log\left(
     \text{Leb}(\Gamma_R(\tau,\epsilon,K,N)\right))+\frac{n}{2}\log N\right),$$
     In \cite{EntPow}, a state { $\tau$ for which} these limits coincide is called \textit{regular}.

   We are now going to concavify %{ \red (in the case p=0)}
    the previous definition in the following way.  
   
   \begin{definition}\label{ExtDef}
  We define the concavified free entropy of a tracial state $\tau\in\mathcal{S}^n_R$ by~:
   $$\tuchi_R(\tau)=\lim_{K\to\infty,\epsilon\to0}\liminf_{N\to\infty}\left(\frac{1}{N^2}\left[\sup_{\sigma\in V_{\epsilon,K}(\tau)}
   \rho_{N,K}(\sigma)\right]+\frac{n}{2}\log N\right),$$
     and  { likewise $\tchi_R(\tau)$ with a $\limsup$ and $\tochi_R(\tau)$ with a limit to $\omega$.}

Finally, we put for $\tau \in \mathcal{S}^n_c$ :
$$\tuchi(\tau)=\sup_{R{\geq \mathcal{R}(\tau)}}\tuchi_R(\tau)$$ and likewise for {$\tchi(\tau)$},
$\tochi(\tau)$.
 \end{definition}
We thus have, as for Voiculescu's free entropy, three variants, but we do not know whether they all coincide. { Note that, since $\{\mu \in P[(H_N^R)^{n}]\ : \tau_\mu\in
 V_{\epsilon,K}(\tau)\}=\cup_{\sigma\in V_{\epsilon,K}(\tau)} \{\mu \in P[(H_N^R)^{n}] \ : \ \tau_\mu\in
\cap_{\eta>0} V_{\eta,K}(\sigma)\},$ we have the alternative formula :
 $$\tuchi_R(\tau)=\lim_{K\to\infty,\epsilon\to0}\liminf_{N\to\infty}\left(\frac{1}{N^2}\left[\sup_{{\mu \in P[(H_N^R)^{n}]\atop \,\tau_\mu\in
 V_{\epsilon,K}(\tau)} }
   \text{Ent}(\mu)\right]+\frac{n}{2}\log N\right)$$
}

We have the fundamental properties~:
\begin{proposition}\label{geneProp}
      The quantity $\tuchi_R(\tau)$ 
    is a concave upper semi-continuous function of
   $\tau$. { So is $\tochi_R(\tau)$. }
Furthermore, {\gree we have}~:
   $$\tuchi_R(\tau)\geq \underline{\chi}_R(\tau),\quad
    \tchi_R(\tau)\geq \chi_R(\tau),\quad
 \tochi_R(\tau)\geq \chi_R^{\omega}(\tau),$$
and $\tchi_R,\tochi_R$ are subadditive: if $\tau_1,\tau_2$ {are the marginal states} giving the noncommutative distributions of $X_1,\ldots,X_m$ and $X_{m+1},\ldots,X_n$ respectively, then
$$\tchi_R(\tau)\leq \tchi_R(\tau_1)+\tchi_R(\tau_2),\quad
\tochi_R(\tau)\leq \tochi_R(\tau_1)+\tochi_R(\tau_2).$$
   \end{proposition}

   \begin{proof}  According to (\ref{Legendre}),
   {\gree we have} :
   $$\rho_{N,K}
   ({ \sigma})=
   \inf_{P\in\Bbb C_{sa}\langle
     X_1,\ldots,X_n\rangle\atop \text{deg}(P)\leq
     K}\left(\log I_N(P)+N^2{ \sigma}(P)\right)$$
   which is therefore  a concave upper semi-continuous
    function of $\sigma$.
   Let $\tau_1$ and $\tau_2$ be states, and let $\sigma_1\in
   V_{\epsilon,K}(\tau_1)$, $\sigma_2\in
   V_{\epsilon,K}(\tau_2)$, then 
   $$\lambda\sigma_1+(1-\lambda)\sigma_2\in
   V_{\epsilon,K}(\lambda\tau_1+(1-\lambda)\tau_2)$$
    therefore by concavity,
   $$\sup_{\sigma \in V_\epsilon(\lambda\tau_1+(1-\lambda)\tau_2)}\rho_{N,K}(\sigma)
   \geq \lambda\rho_{N,K}(\sigma_1)+(1-\lambda)\rho_{N,K}(\sigma_2).$$
    Since this is true for all
   $\sigma_1,\sigma_2$ {\gree we get} :
    $$\sup_{\sigma \in V_\epsilon(\lambda\tau_1+(1-\lambda)\tau_2)}\rho_{N,K}(\sigma)
   \geq \lambda\sup_{\sigma_1 \in V_\epsilon(\tau_1)}\rho_{N,K}(\sigma_1)
   +(1-\lambda)\sup_{\sigma_{ 2} \in V_\epsilon(\tau_2)}\rho_{N,K}(\sigma_2).$$
   
   {\gre The reader may have noted this is also a consequence of the expression of the { left} hand side as the entropy of Csiszar's I-projection on the set of measures having mean in $V_{\epsilon,K}(\lambda\tau_1+(1-\lambda)\tau_2)$.}
   Thus 
   $\sup_{\sigma\in V_{\epsilon,K}(\tau)}
   \rho_{N,K}(\sigma)$ is a concave function of $\tau$, and taking a liminf 
   we see that :
   $$\liminf_{N\to\infty}\left(\frac{1}{N^2}\left[\sup_{\sigma\in V_{\epsilon,K}(\tau)}
   \rho_{N,K}(\sigma)\right]+\frac{n}{2}\log N\right)$$
   is again concave in $\tau$.
   
   It is easy to  check that 
   taking the limit as $\epsilon $ goes to zero gives an upper
   semi-continuous function. Since it is nonincreasing in $K$, the limit as
   $K\to\infty$ is again
   concave and upper semi-continuous.
   
Subadditivity follows from the subadditivity of classical entropy.
Note that { we cannot deduce it} for the $\liminf$ variant, since in general
the inequality
$\liminf (a_n+b_n)\leq \liminf(a_n)+\liminf(b_n)$ fails. Of course if all variants of the free entropy actually coincide, subadditivity would follow in this case.
  \end{proof}
\begin{remark}
We notice that the state of maximal $\tilde 
\chi$ entropy in $\mathcal S_R^n$ is the distribution of a free family of 
arc-sine distributed self-adjoint operators, where the arcsine distribution is
on $[-R,R]$.  It corresponds to taking the limit of barycenters of normalized 
Lebesgue
measure on $(H_N^R)^n$. In particular, this quantity is finite. {\gre (The reader may also be referred to \cite{HP} section 5.6 for this finiteness.)}
\end{remark}

\begin{remark}\label{compVBH}
{ As the referee reminded us, Voiculescu suggested in \cite[section 7.1]{V1} several alternative definitions of free entropy. We discuss here the relation with our definition. The first variant $\chi^{(1)}(\tau)$ has been studied in \cite{B03} and the second variant $\chi^{(2)}(\tau)$ happens to be by definition exactly our $\tchi(\tau).$ The first part of this paper may thus be seen as a study of this suggestion of Voiculescu. 
Recall the definition :

$$\chi^{(1)}(\tau)=\sup_{R\geq \mathcal{R}(\tau)}\lim_{K\to\infty,\epsilon\to0}\limsup_{N\to\infty}\left(\frac{1}{N^2}\left[\sup_{{\mu \in P[(H_N^R)^{n}]\atop {\,E_\mu(|\frac{1}{N}Tr(P)-\tau(P)|)<\epsilon \atop \,\forall P\ \textrm{monomial},\ deg(P)\leq K}} }
   Ent(\mu)\right]+\frac{n}{2}\log N\right)$$
   
   In \cite{B03}, Belinschi proved $\chi^{(1)}(\tau)=\chi(\tau).$ for any $\tau\in \mathcal S_c^n$. We want to point out that the nonlinearity of the condition in  $\frac{1}{N}Tr(P)$ under law $\mu$ is the key why this equality is valid here (as in Hiai's second variant of entropy \cite[section 6]{Hi}). In the variant $\tchi(\tau)$ we only have a condition on $\tau_\mu$, and this allows us  to get a concavification,; this is also what makes it harder to prove equality with $\chi(\tau)$ in the factorial case.
   
   We may also compare our definition with the quantity obtained by  \cite{Hi} using the Legendre transform of  free pressure.
Define,  for $P=P^*\in \C\langle X_1,...,X_n\rangle$ :
   $$\pi_R(P)=\limsup_{N\to \infty} \frac{1}{N^2}\log I_N(P)+\frac{n}{2}\log N.$$
   Hiai defines the entropy by : $$\eta_R(\tau)=\inf_{P=P^*\in \C\langle X_1,...,X_n\rangle }\tau(P)+\pi_R(P).$$

By (\ref{maxent}), for any $P$ monomial of degree less than $K$, $\sigma\in V_{\epsilon,K}(\tau)$, we have : $$\frac{1}{N^2}\rho_{N,K}(\sigma)\leq \frac{1}{N^2}\log I_N(P)+\tau(P)+\epsilon.$$
Thus, taking a supremum, a limsup (or liminf), and then the limit in $\epsilon, K$, {\gree we get}~:
$$\tilde{\chi}_{R}(\tau)\leq \tau(P)+\pi_R(P),$$

so that taking an infemum over $P$ {\gree we also get}~:
$$\tilde{\chi}_{R}(\tau)\leq \eta_R(\tau).$$

We don't know when there is actually an equality, but in the one variable case $(n=1)$, it is known $\eta_R(\tau)=\chi(\tau)$ and thus $\eta_R(\tau)=\chi(\tau)=\tchi(\tau)=\tuchi(\tau)$ for $R$ large enough.
}
\end{remark}

{ In this article, we mainly study $\tilde{\underline{\chi}}_R(\tau)$ instead of  $\tilde{\underline{\chi}}(\tau)$. This is motivated by the following result, really similar to \cite[Proposition 2.4]{V1}.

\begin{proposition}
Consider $\tau \in \mathcal S_c^n$. For any $T>R>\mathcal{R}(\tau)$ we have :
$$\tuchi_T(\tau)=\tuchi_R(\tau)=\tuchi(\tau),\ \ \ \tchi_T(\tau)=\tchi_R(\tau)=\tchi(\tau), \ \ \ \tochi_T(\tau)=\tochi_R(\tau)=\tochi(\tau).$$
\end{proposition}

\begin{proof}
We only prove the $\liminf$ variant, and of course it will suffice to prove for $T>R>\mathcal{R}(\tau)$, $\tuchi_T(\tau)\leq \tuchi_R(\tau)$ (the other inequality is obvious).
Let  $S=\frac{\mathcal{R}(\tau)+R}{2}$.
Define the continuous piecewise linear function $h:[-T,T]\to \R$ by $h(t)=\alpha$ for $t\in [-T,-R]\cup[R,T]$, $h(t)=1$ for $t\in [-S,S]$,   $h(t)=\alpha+(1-\alpha)\frac{t+R}{R-S}$ the linear interpolation for $t\in [-R,-S]$ and $h(t)=\alpha+(1-\alpha)\frac{-t+R}{R-S}, $ with $\alpha=\frac{R-S}{2T-(R+S)}<1$ since $T>R.$ 

In this way, if we define a continuous increasing function $g:[-T,T]\to [-R,R]$
by $g(t)=-R+\int_{-T}^th(s)ds$ we have $g(T)=R$, $g(t)=t$ for $t\in[-S,S]$ and $g'(t)\in[\alpha,1].$ Let also $G:(H_N^T)^n\to (H_N^R)^n$ defined by $G(A_1,...,A_n)=(g(A_1),...,g(A_n)).$ Especially for a state $\tau\in \mathcal S_T^n$, {\gree we get} a state $G_{*}\tau\in \mathcal S_R^n$, so that $\tau_{G_{*}\mu}=G_{*}\tau_\mu,$ defined by : $$(G_{*}\tau)(P(X_1,...,X_n))=\tau(P(g(X_1),...,g(X_n))).$$ 

Fix $\epsilon>0,K\in \N^*,\tau\in \mathcal S_T^n$, we will choose $\delta_1,\delta_2>0$ small enough later.
First, as in the proof of \cite[Proposition 2.4]{V1}, {\gree we get} $0<\epsilon_1<\epsilon/2$, $K_1>K$ such that for any $\sigma\in V_{\epsilon_1,K_1}(\tau)\cap\mathcal S_T^n$ (with $E(X_j,B)$ the spectral projection of the self-adjoint element $X_j$ (computed in its GNS representation) on the set $B\subset \R$) :
$$\sigma(E(X_j,[-T,-S]\cup [S,T]))\leq \delta_1\delta_2,$$
$$\sigma(|g(X_j)-X_j|)\leq \delta_2.$$

This implies $G_{*}\sigma\in V_{\epsilon,K}(\tau\cap\mathcal S_R^n)$, for $\delta_2$ small enough (e.g. $\delta_2<\epsilon/2KT^{K-1}$).

Consider $\mu \in P[(H_N^R)^{n}]$ such that $\tau_\mu\in
 V_{\epsilon_1,K_1}(\tau),$ we can estimate by Chebyshev's inequality :
 
 $$P_\mu(\frac{1}{N}Tr(E(X_j,[-T,-S]\cup [S,T]))\geq \delta_2)\leq \frac{E_\mu(\frac{1}{N}Tr(E(X_j,[-T,-S]\cup [S,T]))}{\delta_2}\leq \delta_1.$$

 We can also compute $\frac{dG_*\mu}{dLeb}=(\frac{d\mu}{dLeb}\circ G^{-1})\times |det(Jac(G^{-1}))|$. If we write $\partial g$ the two variable function $\partial g(A,B)=(g(A)-g(B))/(A-B), A\neq B$ extended by $\partial g(B,B)=g'(B)$ on the diagonal, the jacobian of $g$ is given by $\partial g$ applied by functional calculus so that :
 
 $$Ent(G_*\mu)=Ent(\mu)+\sum_jE_\mu(\frac{1}{2}(Tr\otimes Tr)(\log |\partial g(X_j\otimes 1,1\otimes X_j)|^2)).$$
In the proof of \cite[Proposition 2.4]{V1}, Voiculescu showed that, 
 for a matrix $X_j\in (H_N^T)^{n}$ such that $\frac{1}{N}Tr(E(X_j,[-T,-S]\cup [S,T]))\leq \delta_2$, the positive determinant of the jacobian of $g$ is bounded bellow so that :  $$\left|\frac{1}{2}(Tr\otimes Tr)(\log |\partial g(X_j\otimes 1,1\otimes X_j)|^2)\right|\leq(N+N^2-(N(1-\delta_2))^2)|\log \alpha| .$$
 Moreover for any matrix $X_j\in (H_N^R)^{n}$, {\gree we have} : $\left|\frac{1}{2}(Tr\otimes Tr)(\log |\partial g(X_j\otimes 1,1\otimes X_j)|^2)\right|\leq N^2|\log \alpha| .$
 
 As a consequence, we get :
 
 $$Ent(G_*\mu)\geq Ent(\mu)-n(N+N^2(2\delta_2-\delta_2^2))|\log \alpha|-n\delta_1N^2|\log \alpha|.$$
 
 Taking suprema and liminf, {\gree we get} :
 \begin{align*}&\liminf_{N\to\infty}\left(\frac{1}{N^2}\left[\sup_{{\nu \in P[(H_N^R)^{n}]\atop \,\tau_\nu\in
 V_{\epsilon,K}(\tau)} }
   \rho_{N,K}(\sigma)\right]+\frac{n}{2}\log N\right)\\ &\geq\liminf_{N\to\infty}\left(\frac{1}{N^2}\left[\sup_{{\mu \in P[(H_N^T)^{n}]\atop \,\tau_\mu\in
 V_{\epsilon_1,K_1}(\tau)} }
   \rho_{N,K}(\sigma)\right]+\frac{n}{2}\log N\right) + n(2\delta_2-\delta_2^2)\log \alpha+n\delta_1\log \alpha.
   \end{align*}

 Since $\delta_1,\delta_2$ can be made arbitrarily small choosing $\epsilon_1,K_1$, {\gree we get} the desired inequality.
\end{proof}

}
\section{A preliminary separation result}
In order to prove that Voiculescu's entropy coincides with its modification on extremal states, we will need a separation result. We gather here references to the literature.
Recall that for $K$ a convex subset of the dual $E^{*}$ of a { complex} topological vector space, an $x\in E$ is said to expose $f$ in  $K$ if $f\in K$ and $\Re g(x) < \Re f(x)$ for all $g\in K$ other than $f$. Those $f$ which are so exposed by elements of $E$ are weak-* exposed points of $K$. We now state  a result of Sidney \cite{S} (attributed by Asplund to Bishop  in the Banach space case)

\begin{proposition}
Let $E$ be a separable Frechet space and $K$ a non-empty convex weak* compact subset of its topological dual $E^{*}$. Then $K$ is the weak* closed convex hull of the set of its weak* exposed points (this set is thus non empty).
\end{proposition}

Since it is proved in \cite{D} that $\mathcal S_R^n$ is a Poulsen simplex, we will use the following result \cite{L} of homogeneity.

\begin{proposition}
                     Let $S_{1}$ and $S_{2}$ be metrizable simplices with
$\overline{Ext\ S_{i}}= S_{i}$ (i.e. Poulsen simplices), for i = 1,2. Let $F_{i}$ be a proper closed face of $S_{i}$, i = 1,2,
and let $\varphi$ be an affine homeomorphism which maps $F_{2}$ onto $F_{1}$.
    Then $\varphi$ can be extended to an affine homeomorphism which
maps $S_{2}$ onto $S_{1}$.
\end{proposition}

Applying those two results, the second to move any extremal point to a weak-* exposed point, which exists via the first result, one easily gets :

 \begin{proposition}
                      Let $E$ be a separable Fr\'{e}chet space and $K$ a non-empty convex weak* compact subset of its topological dual $E^{*}$, which is a Poulsen simplex. Then any extreme point of $K$ is a weak* exposed point.
\end{proposition}

\begin{corollary}\label{sep}
Let $\tau$ be an extremal state in $\mathcal S_R^n,\ { n>1}$, and $\epsilon>0$. For any $\eta>0$, there exists a self adjoint polynomial 
$Q_{\eta}\in\Bbb C\langle X_1,\ldots, X_n\rangle$ 
such that for every $\sigma\in \mathcal S_R^n$ {\gree we have}~:
$$\tau(Q_{\eta})>\sigma(Q_{\eta})-\eta,$$ 
and for all $\sigma\notin V_{\epsilon,K}(\tau)$ one
has~: $$\sigma(Q_{\eta})<\tau(Q_{\eta})-1.$$
\end{corollary}
\begin{proof}
Take $\eta<1/2$. Since $\tau$ is weak-* exposed, first take $Q$ in $*_{i=1}^nC([-R,R])$ exposing it in $\mathcal S_R^n$, one can assume $Q$ self adjoint.
 After multiplication by a scalar one can assume, since $V^{c}=V_{\epsilon,K}(\tau)^{c}\cap \mathcal S_R^n$ is a compact set, that $\sup_{\sigma\in V^{c}}\sigma(Q)\leq \tau(Q)-2$. Let $Q_{\eta}$  be a self-adjoint polynomial such that $||Q-Q_{\eta}||_{R}\leq \eta/2$. For any state $\sigma$  {\gree we have} $|\sigma(Q_{\eta})-\sigma(Q)|\leq \eta/2$,
thus if  $\sigma\neq\tau$: $$\sigma(Q_{\eta})\leq\sigma(Q)+\eta/2<\tau(Q)+\eta/2\leq \tau(Q_{\eta})+\eta$$
and :$$\sup_{\sigma\in V^{c}}\sigma(Q_{\eta})\leq \tau(Q_{\eta})-2+\eta<\tau(Q_{\eta})-1.$$
\end{proof}

\section{Extremal states}

We first prove a concentration lemma.
\begin{lemma}\label{concentration}
If $\tau$ is an extremal state in $\mathcal S_R^n,\ { n>1}$, then for any $\eta,\epsilon, K>0$ there exists $\delta,L>0$ such that, for any 
  probability measure  $\mu$ on $(H_N^R)^n$,
 whose barycenter is in
 $V_{\delta,L}(\tau)$,  {\gree we have}~:
 $$\mu(\Gamma_R(\tau,\epsilon,K,N))\geq 1-\eta.$$
\end{lemma}
\begin{proof}
Let $\eta\in ]0,1/4[$. Then, by Corollary \ref{sep},   we
can find some self adjoint polynomial 
$Q_{\eta}\in{\bf C}\langle X_1,\ldots, X_n\rangle$ 
such that for every $\sigma\in \mathcal S_R^n$ {\gree we have}~:
$$\tau(Q_{\eta})>\sigma(Q_{\eta})-\eta/2$$ 
and for all $\sigma\notin V_{\epsilon,K}(\tau)$ one
has~: $$\sigma(Q_{\eta})<\tau(Q_{\eta})-1.$$
%See e.g. \cite{Bou} II \textsection  7.1.

Let us now choose $L=\text{deg}(Q_{\eta})$, and $\delta$ small enough so that
 for all $\sigma\in V_{\delta,L}(\tau)$ {\gree we have}~:
 $$|\tau(Q_{\eta})-\sigma(Q_{\eta})|<\eta/2.$$
 If  $\mu$ is a probability measure on $(H_N^R)^n$
 whose barycenter  $\tau_{\mu}$ is in
 $V_{\delta,L}(\tau)$ then {\gree we have}
  \begin{align*}
 \tau(Q_{\eta})-\eta/2&\leq\tau_{\mu}(Q_{\eta})\\ &= 
 \int_{\Gamma_R(\tau,\epsilon,K,N)}\frac{1}{N}Tr(Q_{\eta})d\mu+
 \int_{(H^R_N)^{n}\setminus
 \Gamma_R(\tau,\epsilon,K,N)}\frac{1}{N}Tr(Q_{\eta})d\mu\\
 &\leq \mu(\Gamma_R(\tau,\epsilon,K,N))(\tau(Q_{\eta
 })+
 \eta/2)\\&\quad+
 \mu((H^R_N)^{n}\setminus\Gamma_R(\tau,\epsilon,K,N))(\tau(Q_{\eta%,A
 })-1)
 \\ &\leq  \tau(Q_{\eta})+\eta/2- (1-\mu(\Gamma_R(\tau,\epsilon,K,N))).
  \end{align*}
  
 Therefore
 $$\mu(\Gamma_R(\tau,\epsilon,K,N))\geq 1-\eta.$$
  
\end{proof}

\begin{proposition}\label{main}
For any factor state $\tau$ in $\mathcal S_R^n,\ {n>1}$~:
$$\tuchi_R(\tau)={\gre \underline{\chi}_R(\tau)}.$$
Likewise $\tochi_R(\tau)={\red \chi^\omega_R(\tau)},\ \tchi_R(\tau)={\red \chi_R(\tau)}.$
\end{proposition}
\begin{proof} Consider an extremal state $\tau\in\mathcal S_R^{n}$, and $\eta,\epsilon,K>0$.
We can choose $\delta,L$ as in Lemma \ref{concentration}, so that
we can  estimate the
 entropy of $\mu$ using (\ref{ent}) { and the variations of $x\mapsto x\log x+(1-x)\log(1-x)$}~:%, (\ref{entnu}). 
 {
 \begin{align*}
 \text{Ent}(\mu)\leq &
 \log \text{Leb}\left(\Gamma_R(\tau,\epsilon,K,N)\right)
 +\mu(\Gamma_R(\tau,\epsilon,K,N)^c)\log\frac{\text{Leb}\left((H_N^R)^n\right)}{\text{Leb}\left(\Gamma_R(\tau,\epsilon,K,N)\right)}\\ &-\eta\log
 (\eta)-(1-\eta)\log(1-\eta)
 \\ \leq  &
 (1-\eta)\log \text{Leb}\left(\Gamma_R(\tau,\epsilon,K,N)\right)
 +\eta\log\text{Leb}\left((H_N^R)^n\right)\\ &-\eta\log
 (\eta)-(1-\eta)\log(1-\eta)
 \end{align*}
 }
 
 %$$
 %\begin{array}{rcl}
 %\text{Ent}(\mu)\leq &
 %(1-\eta)\log \text{Leb}\left(\Gamma_R(\tau,\epsilon,K,N)\right)
 %+\eta\log\text{Leb}\left((H_N^R)^n\right)\\ &-\eta\log
 %(\eta)-(1-\eta)\log(1-\eta)
 %\end{array}
 %$$

 This inequality holds for all probability measures with barycenter in
 {\red $V_{\delta,L}(\tau)$}, therefore the right hand side is a majorant of
 $\sup_{\sigma\in V_{\delta,L}(\tau)}\rho_{N,L}(\sigma)$.
 Now multiply both sides of this inequality by $1/N^2$, add $\frac{n}{2}\log N$
 and take $\liminf$ {\red(or $\limsup$ or a limit to $\omega$)} then infimum over, successively, $\delta,L,\epsilon,K,$ to get the result.
 \end{proof}

\section{Orbital free entropy and freeness in case of additivity of entropy}
\subsection{Motivation}
In this section, we extend the definition of orbital free entropy of \cite{HiaiMUeda}  to not necessarily hyperfinite multivariables.
Let us explain the main ideas before entering into technical details. Orbital entropy aims at measuring the lack of freeness in the same way that relative entropy of a measure $\mu$ with respect to the tensor product of its marginals (also called mutual information) does measure the lack of independence in the classical case.
In the non-microstate context, Voiculescu first introduced in \cite{V6} a notion of mutual information $i^{*}(W^{*}(X_1,...,X_m),W^{*}(X_{m+1},...,X_{m+n}))$ measuring this lack of freeness using conjugation by a free unitary brownian motion and proved, using this tool, that additivity of non-microstate free entropy implies freeness. In the microstate context, \cite{HiaiMUeda}   defined   $\chi_{orb}(X_1;...;X_n)$, which  measures the lack of freeness of $W^*(X_1),...,W^*(X_n)$ (and a variant where the $W^*(X_i)$ are replaced by hyperfinite algebras) relying on the fact that (at least at the level of measure spaces) the space of hermitian matrices can be factored into eigenbasis and eigenvalues, allowing to build microstates in a product of unitary groups. This idea however breaks down when one tries to replace $X_i$ by sets of variables generating non hyperfinite algebras, since in this case there does not exist a good description of the microstates.  
Our idea here is to overcome this lack of a  microstates model by using entropies of measures  instead of volumes of microstates.
At this point,  we have several possible candidates for a generalization. We will use one of them in this section, in order to reach our goal, the result that additivity of free entropy implies freeness. We will explore further possibilites, in order to lay the ground for future investigations, in the last section. 
{\red Finally, note that we prove this result about additivity only  for extremal states. It seems likely that  for nonextremal states freeness should be replaced by a kind of freeness with amalgamation with respect to some commutative central algebra, but we do not investigate this in the present paper. }
 
\subsection{The orbital free entropy of Hiai, Miyamoto and Ueda}
We consider finite sets of non-commutative random variables 
$\mathbf{X_i}=\{X_{i1},...,X_{iP_i}\}$ for $i=1,\ldots,n$, and $\bar n=\sum_i P_i$, { $\tilde{P}=\max_i P_i$}
with joint non-commutative (tracial) distribution  $\tau_{\mathbf{X_1};...;\mathbf{X_n}}\in \mathcal S_{ c}^{\bar n}$.
When each set $\mathbf{X_i}$ generates a hyperfinite algebra, 
Hiai, Miyamoto and Ueda \cite{HiaiMUeda} defined  orbital free entropy
$\chi_{orb}(\mathbf{X_1};...;\mathbf{X_n})$.
Let us recall their definition.
Let  $(\Xi_i(N))_{i=1\ldots n}; N\to\infty$ 
be a sequence of matrix sets of size $N$ ($\Xi_i=\{\xi_{i1},\ldots,\xi_{iP_i}\}$)
 which approximates  $(\mathbf{X_i})_{i=1\ldots n}$ in mixed moments as $N\to\infty$.
For $U\in U(N)$ we denote $U\Xi_i(N)U=\{U\xi_{i1}U^*,\ldots,U{ \xi_{iP_i}}U^*\}$.
Let $$\Gamma_{orb}(\mathbf{X_1},...,\mathbf{X_n}:\Xi_1 (N),...\Xi_n (N),N,K,\epsilon)$$ be the set of 
$(U_1,\ldots,U_n)\in U(N)^n$ such that the conjugated sets $(U_i\Xi_i(N)U_i^*)_{i=1\ldots n}$ approximate the mixed moments of $(\mathbf{X_i})_{i=1\ldots n}$ up to an error of $\epsilon$ and for degrees less than $K$ or, in other words, such that
$\tau_{U_1\Xi_1 (N)U_1^*,\ldots,U_n\Xi_n (N)U_n^*}\in V_{\epsilon,K}(\tau_{\mathbf{X_1};...;\mathbf{X_n}})$.
 Let $\mathcal{ H}_N^n$ be the Haar measure on $U(N)^n$ { (which, in the sequel, we will always assume normalized to be a probability)}, and 
\begin{equation}\label{orbit}
\gamma_{N,\Xi(N),\epsilon,K}= {\mathcal H}_N^n(\Gamma_{orb}(\mathbf{X_1},...,\mathbf{X_n}:\Xi_1(N),...,\Xi_n (N),N,K,\epsilon))
\end{equation}
then the orbital free entropy is defined as~:
$$\chi_{orb}(\mathbf{X_1},\ldots,\mathbf{X_n})=\lim_{\epsilon\to 0,K\to\infty}\limsup_{N\to \infty} \frac{1}{N^2}\log \gamma_{N,\Xi(N),\epsilon,K}.$$
It is proved in \cite{HiaiMUeda}, Lemma 4.2, that this quantity does not depend on the chosen sequence $\Xi(N)$. This relies on Jung's Lemma 
 \cite{Jung}, (see also Lemma 1.2 in \cite{HiaiMUeda}) which we recall here for future reference.
\begin{lemma}Let $\tau=\tau_{X_1,...,X_m}$ where the variables
$X_1,\ldots,X_m$ generate a hyperfinite algebra. Denote by $\Vert.\Vert_p$ the $p$-norm associated with $\tau$.
For every $\epsilon>0$ there exists $L,\delta$ such that,
for every $\Xi=(\xi_1,\ldots,\xi_m)$ and $\Xi'=(\xi'_1,\ldots,\xi'_m)$ in $({H}_N)^n$, satisfying 
$\tau_{\Xi},\tau_{\Xi'}\in V_{\delta,L}(\tau)$, there exists some unitary
$U\in U(N)$ such that
$$\Vert U\xi_iU^*-\xi'_i\Vert_p<\epsilon,\ \text{for}\ i=1,\ldots,n$$
\end{lemma}
Furthermore, Hiai, Miyamoto and Ueda proved 
that free orbital entropy depends only on the $W^*$-algebras
$W_i=\mathbf{X_i}''$ generated by each set, i.e.
$$\chi_{orb}(\mathbf{X_1};...;\mathbf{X_n})=\chi_{orb}(\mathbf{Y_1};...;\mathbf{Y_n}),$$ for any other choice of finite  sets $\mathbf{Y_1};...;\mathbf{Y_n}$ such that
$W_i=\mathbf{Y_i}''$ (note that one does not assume that $\mathbf{X_i},\mathbf{Y_i}$ contain the same number of elements). Also they proved the formula relating orbital free entropy to Voiculescu's free entropy:
\begin{equation}\label{add}\chi(\mathbf{X_1}\cup...\cup\mathbf{X_n})= \chi_{orb}(\mathbf{X_1};...;\mathbf{X_n})
+ \chi(\mathbf{X_1})+...+\chi(\mathbf{X_n}),
\end{equation}
and proved that, for a set with finite free entropy, additivity of free entropy, {
i.e. \begin{equation}\chi(\mathbf{X_1}\cup...\cup\mathbf{X_n})= \chi(\mathbf{X_1})+...+\chi(\mathbf{X_n}),
\end{equation}}
which is
equivalent to
 $\chi_{orb}(\mathbf{X_1};...;\mathbf{X_n})=0$ by (\ref{add}),
holds if and only if $\mathbf{X_1},\ldots,\mathbf{X_n}$ are free.

\subsection{Orbital free entropy for arbitrary multivariables}
In the following, we give a definition of $\tchi_{orb}(\mathbf{X_1};...;\mathbf{X_n})$
for arbitrary finite sets $\mathbf{X_1};...;\mathbf{X_n}$, which coincides with the previous definition when the sets of multivariables are hyperfinite and $\tau_{\mathbf{X_1};...;\mathbf{X_n}}$ is a factor state.

Let  $\mu \in P(H_N^R)^{\bar n}$ be a  (Borel) probability measure on $(H_N^R)^{\bar n}=\prod_i(H_N^R)^{P_i}$, considered as the joint distribution of  sets of random matrices $\mathbf{M_1};...;\mathbf{M_n}$, with
$\mathbf{M_i}=\{M_{i1},\ldots,M_{iP_i}\}$.
We denote by $U\mu$  the probability measure on $(H_N^R)^{\bar n}$, obtained by conjugating the sets $\mathbf{M_i}$ by independent Haar unitaries from $U(N)$, i.e. $U\mu$ is the joint distribution of the sets $U_i \mathbf{M_i}U_i^*=\{U_i M_{i1}U_i^*,\ldots,U_iM_{iP_i}U_i^*\}$, where $U_1,\ldots, U_n$ are independent unitary matrices, all distributed according to (normalized) Haar measure on  $U(N)$.
Equivalently, if $$\Phi_N:U(N)^n\times (H_R)^{\bar n}\to (H_R)^{\bar n}$$ is the map given by conjugation:
$$(U_i,\mathbf{X_i})_{i=1,\ldots,n}\mapsto (U_i\mathbf{X_i}U_i^*)_{i=1,\ldots,n}$$
then { $U\mu$ is given by the pushforward measure} :
$$U\mu=\Phi_{N*}(\mathcal{H}_N^n\otimes\mu).$$

 \begin{definition}
Let $\mathbf{X_1};...;\mathbf{X_n}$ be finite sets of noncommutative random variables as above,
  their \textit{\textbf{orbital entropy}} is defined as~:
    $$\tchi_{orb}(\mathbf{X_1};...;\mathbf{X_n})
                  =\sup_{R{ \geq \mathcal{R}(\tau_{\mathbf    {X_1},...,\mathbf{X_n}})}}\lim_{K\to\infty,\epsilon\to 0}
                \limsup_{N\to\infty}
\left(\frac{1}{N^2}\sup_{{\mu \in P(H_N^R)^{\bar n}\atop \,\tau_\mu\in
 V_{\epsilon,K}(\tau_{\mathbf    {X_1},...,\mathbf{X_n}})} }
\text{Ent}( \mu |U \mu)\right).$$
Similarly we define the $\liminf $ and ultrafilter variants
$\underline{\tchi}_{orb}$ and $\tchi_{orb}^\omega$.
     \end{definition}
   Note that, in this definition, limits in $\epsilon,K$ are actually infima. 

{\bl Recall from \cite[Def 3.1]{EntPow} that a state is said to have \textit{finite-dimensional approximants} if for every $K,\epsilon$ there exists $N_0$ such that for $N\geq N_0$, $\Gamma_R(\tau,\epsilon,K,N)\neq \emptyset$}.

\begin{theorem}\label{prop}
The orbital free   entropy satisfies the following properties.

\medskip

{\sl $(1)$  (Negativity)}
$$\tchi_{orb}(\mathbf{X_1};...;\mathbf{X_n})\leq 0.$$

\medskip

{\sl $(2)$  (Vanishing for one multivariable) }
$$\tchi_{orb}(\mathbf{X})=0,$$ for any single multivariable 
${\mathbf X}=\{X_1,\ldots,X_m\}$ {\bl having finite-dimensional approximants}.

\medskip

{\sl $(3)$ (Monotonicity)}
$$ \tchi_{orb}(\mathbf{X_1};...;\mathbf{X_n})\leq \tchi_{orb}( \mathbf{Y_1};...; \mathbf{Y_n}),$$
if ${\mathbf Y_i}\subset {\mathbf X_i}$ for $1\leq i\leq n$.

\medskip

{\sl $(4)$ (Subadditivity)}
$$\tchi_{orb}(\mathbf{X_1};...;\mathbf{X_m};\mathbf{X_{m+1}};...;\mathbf{X_n})\leq \tchi_{orb}(\mathbf{X_1},...,\mathbf{X_m})+ \tchi_{orb}(\mathbf{X_{m+1}};...;\mathbf{X_n}).$$

\medskip

{\sl $(5)$ (Connection with free entropy)}
 $$\tchi(\mathbf{X_1}\cup...\cup\mathbf{X_n})\leq \tchi_{orb}(\mathbf{X_1},...,\mathbf{X_n})+ \tchi(\mathbf{X_1})+\ldots+\tchi(\mathbf{X_n}).$$

 \medskip

{\sl $(6)$ (Agreement with previous definition)}

 Assume  $\mathbf{X_i}$ are hyperfinite multivariables and let $\chi_{orb}(\mathbf{X_{1}};...;\mathbf{X_n})$ denote the orbital free entropy of \cite{HiaiMUeda},
then 
$$\chi_{orb}(\mathbf{X_{1}};...;\mathbf{X_n})\leq \tchi_{orb}(\mathbf{X_{1}};...;\mathbf{X_n}).$$
Moreover, if $\tau_{\mathbf{X_{1}},...,\mathbf{X_{n}}}$ is extremal then
 $$\chi_{orb}(\mathbf{X_{1}};...;\mathbf{X_n})=\tchi_{orb}(\mathbf{X_{1}};...;\mathbf{X_n}).$$  

 \medskip

 {\sl $(7)$ (Alternative microstates formula in the extremal case)}

If $\tau_{\mathbf{X_{1}},...,\mathbf{X_{n}}}$ is extremal then
 $$\tchi_{orb}(\mathbf{X_{1}};...;\mathbf{X_n})=\sup_R\lim_{K\to\infty,\epsilon\to 0}
                \limsup_{N\to\infty}
\sup_{{\Xi,\tau_\Xi\in
 V_{\epsilon,K}(\tau_{\mathbf    {X_1},...,\mathbf{X_n}})} }
\left(\frac{1}{N^2}\log\gamma_{N,\Xi,\epsilon,K}\right).$$  

 \medskip

{\sl $(8)$ (Dependence on algebras)}

If $\mathbf{X_{1}},\dots,\mathbf{X_{n}}$, $\mathbf{Y_{1}},\dots,\mathbf{Y_{n}}$ are  multi-variables such that
$\mathbf{Y_{i}}\subset W^*(\mathbf{X_{i}})$ for $1\le i\le n$, then
$$
\tchi_{orb}(\mathbf{X_{1}},\dots,\mathbf{X_{n}})\leq\tchi_{orb}(\mathbf{Y_{1}},\dots,\mathbf{Y_{n}}).
$$
In particular, $\tchi_\mathrm{orb}(\mathbf{X_{1}},\dots,\mathbf{X_{n}})$ 
 depends only upon
$W^*(\mathbf{X_{1}}),\dots,W^*(\mathbf{X_{n}}).$
 
\medskip

{\sl $(9)$ (Orbital Talagrand's inequality and Characterization of Freeness)}

 For $\tau=\tau_{\mathbf{X_{1}},...,\mathbf{X_{n}}}$ extremal, let $\tau_{free}=\tau_{\mathbf{X_{1}}}*\dots*\tau_{\mathbf{X_{n}}}$ the free product of its marginals, then :
$$d_{W}(\tau,\tau_{free})\leq 4R\sqrt{-\tilde{ P}\tchi_{orb}(\mathbf{X_{1}};...;\mathbf{X_{n}})},$$
where $d_W$ is the $2$-Wasserstein distance of \cite{BV}.
 As a consequence, if $\tau$ { is extremal {and} has finite-dimensional approximants}, then $\tchi_{orb}(\mathbf{X_{1}};...;\mathbf{X_{n}})=0$ if and only if $\tau=\tau_{free}$.
 \end{theorem}
 \begin{corollary}\label{additivity}
If $\chi^\omega(\mathbf{X_{1}}\cup...\cup\mathbf{X_{n}})>-\infty$, then
 $$\chi^\omega(\mathbf{X_{1}}\cup...\cup\mathbf{X_{n}})=\chi^ \omega(\mathbf{X_{1}})+...+\chi^\omega(\mathbf{X_{n}})$$ 
 if and only if
$\mathbf{X_{1}},...,\mathbf{X_{n}}$ are free. The only if part also holds for the limsup variant $\chi$.
 \end{corollary}
\begin{proof}[Proof of corollary]%{\sl Proof of corollary.}
Assume $\mathbf{X_{1}}\cup...\cup\mathbf{X_{n}}$  has finite entropy , and
$$\chi(\mathbf{X_{1}},...,\mathbf{X_{n}})=\chi(\mathbf{X_{1}})+...+\chi(\mathbf{X_{n}}).$$

{By finiteness of Voiculescu's entropy we  know that $\tau_{\mathbf{X_{1}}\cup...\cup\mathbf{X_{n}}}$ is an extremal state and, by proposition \ref{main}, $\chi(\mathbf{X_{1}},...,\mathbf{X_{n}})=\tchi(\mathbf{X_{1}},...,\mathbf{X_{n}})$.

Assume also for contradiction $\tchi_{orb}(\mathbf{X_{1}},\dots,\mathbf{X_{n}})<0$.  From (5) of Theorem \ref{prop}, {\gree we get} :
$$\chi(\mathbf{X_{1}})+...+\chi(\mathbf{X_{n}})< \tchi(\mathbf{X_{1}})+...+\tchi(\mathbf{X_{n}}).$$
By the general inequality in proposition \ref{geneProp}, there exists an $i$ with $\chi(\mathbf{X_{i}})< \tchi(\mathbf{X_{i}}).$ By the end of remark \ref{compVBH}, the set $\mathbf{X_{i}}$ contains at least two variables, so that by proposition \ref{main} again, $\tau_{\mathbf{X_{i}}}$ cannot be extremal, which implies $\chi(\mathbf{X_{i}})=-\infty$ by Voiculescu's result \cite{Voi3}, a contradiction with $\chi(\mathbf{X_{1}}\cup...\cup\mathbf{X_{n}})>-\infty.$

We thus deduce, using (1) of Theorem \ref{prop}, $\tchi_{orb}(\mathbf{X_{1}},\dots,\mathbf{X_{n}})=0.$}

Then $\mathbf{X_{1}},...,\mathbf{X_{n}}$ are free by point (9) of the same theorem.  The ultrafilter variant is similar. 
The converse statement is due to Voiculescu \cite{EntPow}.
\end{proof}

\medskip

\begin{proof}[Proof of Theorem \ref{prop}]
In the following we say that $\mu$ is an approximating measure for $\tau$ if $\tau_\mu$
 belongs to $V_{\epsilon,K}(\tau)$ for some $\epsilon,K>0$.

\medskip

{\sl $(1)$}
Negativity follows from the negativity of relative entropy. 

\medskip

{\sl $(2)$} For a single multivariable, if $\mu$ is an approximating measure, then $\nu=U\mu$ also approximates with the same precision, and obviously $U\nu=\nu$, therefore
$\text{Ent}(\nu|U\nu)=0$. 

\medskip

{\sl $(3)$}
If ${\mathbf Y_i}\subset {\mathbf X_i}$ for $1\leq i\leq n$ and $\mu$ is an approximating measure
for the ${\mathbf X_i}$, then  its  image by the projection map $q$ on the marginal distribution of the ${\mathbf Y_i}$ is an approximating measure for the 
${\mathbf Y_i}$, furthermore $qU\mu=Uq\mu$, therefore
 by (\ref{push}) {\gree we have} 
$$\text{Ent}(\mu|U\mu)\leq \text{Ent}(q\mu|qU\mu)=\text{Ent}(q\mu|Uq\mu),$$ and taking limits gives  the required inequality. 
 
\medskip

{\sl $(4)$}   If $\mu$ is an approximating measure for 
$\mathbf{X_1};...;\mathbf{X_n}$ let $\mu_1$ and $\mu_2$ denote the marginal distributions of  
$\mathbf{X_1};...;\mathbf{X_m}$ and $\mathbf{X_{m+1}};...;\mathbf{X_n}$, then
$U\mu_1$ and $U\mu_2$ are the marginal distributions under $U\mu$, therefore, by subadditiviy of relative entropy:
$$\text{Ent}(\mu|U\mu)\leq
\text{Ent}(\mu_1|U\mu_1)+\text{Ent}(\mu_2|U\mu_2).$$
The inequality follows by taking limits.

\medskip

{\sl $(5)$} 
Let $\mu$ be an approximating measure on
 $(H_N^R)^{\bar n}=(H_N^R)^{{P_1}}\times\ldots\times (H_N^R)^{{P_n}}$, with finite entropy, and consider the action of $U(N)^n$ by conjugation on $(H_N^R)^{\bar n}$, then $U\mu$ is the average of $(U_1,\ldots,U_n)\cdot\mu$ with respect to Haar measure on $U(N)^n$. Let $f$ be the density of $\mu$ with respect to Lebesgue measure on $(H_N^R)^{\bar n}$, then $f_U$, the density of $U\mu$ is the average of $f((U_1,\ldots,U_n)\cdot)$ with respect to Haar measure. It follows that~:
\begin{eqnarray*}\text{Ent}(\mu)&=&-\int_{(H_N^R)^{\bar n}} f\log f\, dM
\\&=&-\int_{(H_N^R)^{\bar n}} \frac{f}{f_U}\log \frac{f}{f_U} f_U\, dM-\int_{(H_N^R)^{\bar n}}
f\log f_U dM
\\&=&\text{Ent}(\mu|U\mu)-\int_{(H_N^R)^{\bar n}}f\log f_U dM 
\\&=&\text{Ent}(\mu|U\mu)-\int_{(H_N^R)^{\bar n}}f_U\log f_U dM \quad\text{  by $U(N)^n$ invariance of $dM$}
\\&=&\text{Ent}(\mu|U\mu)+\text{Ent}(U\mu).
\end{eqnarray*}
Now we can use the subbadditivity of $\text{Ent}(U\mu)$ with respect to the projections on the spaces $(H_N^R)^{P_i}$, which gives
$$\text{Ent}(U\mu)\leq \text{Ent}(p_1U\mu)+\ldots +\text{Ent}(p_nU\mu)$$
Letting $N\to\infty, K\to\infty,\epsilon\to 0$ gives the required inequality.

\medskip

{\sl $(6)$} Let $\Xi(N)$ be an approximating sequence, as in the definition of (hyperfinite) orbital free entropy.
Let $\nu_{ \Xi(N)}$ be the probability measure obtained by restricting  $\mathcal {H}_N^n$ to 
$\Gamma_{orb}(\mathbf{X_1},...,\mathbf{X_n}:\Xi_1{(N)},...\Xi_n{ (N)},N,K,\epsilon)$ and normalizing (if the orbital entropy is finite and $N$ is sufficiently large, this measure is well defined), then (recall (\ref{orbit}))
$$\log \gamma_{N,\Xi(N),\epsilon,K}=\text{Ent}(\nu_{ \Xi(N)}|{\mathcal {H}_N^n}).$$
Let $\Psi_{ \Xi(N)}:U(N)^n\to (H_R)^{\bar n}$ the map given by conjugation~:
$$(U_i)_{i=1,\ldots,n}\mapsto (U_i\Xi_i(N)U_i^*)_{i=1,\ldots,n}$$ {\gree we have}
$U\Psi_{\Xi(N)*}(\nu_{ \Xi(N)})=\Psi_{ \Xi(N)*}(\mathcal {H}_N^n)$ therefore, by (\ref{push})
$$\frac{1}{N^2}\text{Ent}(\nu_{ \Xi(N)}|\mathcal {H}_N^n)\leq \frac{1}{N^2}\text{Ent}(\Psi_{ \Xi(N)*}(\nu_{ \Xi(N)})|U\Psi_{ \Xi(N)*}(\nu_{ \Xi(N)})).$$ Since $\Psi_{ \Xi(N)*}(\nu_{ \Xi(N)})$ is an approximating measure for $\mathbf{X_1},...,\mathbf{X_n}$, the right hand side, after taking  limits in $N,\epsilon,K$, is bounded by 
$\tchi_{orb}(\mathbf{X_1},...,\mathbf{X_n})$. The inequality 
$\chi_{orb}\leq \tchi_{orb}$ follows.

Let us now assume that $\tau:=\tau_{\mathbf{X_{1}},...,\mathbf{X_{n}}}$ is extremal. 
Fix $\eta,\epsilon>0$ and an integer $K>0$. Using  Jung's Lemma, and following the proof of Lemma 4.2 in \cite{HiaiMUeda}, we can take $\delta{\leq \epsilon/2},L{\geq K}$ such that,
for all families of sets  $(\Theta_i)_{i=1,\ldots,n}$ of $N\times N$ hermitian matrices such that  for all $i$ {\gree we have} $\tau_{(\Theta_i)}\in V_{\delta,L}(p_i\tau)$
(a fortiori if $\tau_{(\Theta_i)_{i=1,\ldots,n}}\in V_{\delta,L}(\tau)$) {\gree we have}, for $N$ large enough~:
\begin{equation}\label{thetaxi}
\gamma_{N,\Theta,\epsilon/2,K} \leq \gamma_{N,\Xi(N),\epsilon,K}.
\end{equation}
Note also the elementary equality for any $U_1,...,U_n$ unitaries coming from invariance of { the} Haar measure :
\begin{equation}\label{invtheta}
\gamma_{N,\Phi_N(U_1,...,U_n,\Theta),\epsilon/2,K}=\gamma_{N,\Theta,\epsilon/2,K}.
\end{equation}
Then using Lemma \ref{concentration}, if we take $\delta'>0$  sufficiently small and $L'$ sufficiently large,  for any measure $\mu$ on 
$({H}_N^R)^{\bar n}$ such that $\tau_\mu\in V_{\delta',L'}(
\tau)$, {\gree we get}~:
$$\mu(\Gamma_R(\tau,\delta,L,N))\geq 1-\eta.$$
Therefore, by (3.3),
\begin{equation}\label{subadd1}\text{Ent}(\mu|U\mu)\leq (1-\eta)\log \left[U\mu(\Gamma_R(\tau,\delta,L,N))\right]-f(\eta),
\end{equation}
with $f(\eta)=\eta\log\eta+(1-\eta)\log(1-\eta)$.

 Let  $U\Gamma_R(\tau,\delta,L,N)=\{\Theta \ | \ \exists (U_1,..,U_n)\ \Phi_N(U_1,...,U_n,\Theta)\in \Gamma_R(\tau,\delta,L,N)\}.$
The measure $U\mu$ is the image of $\mathcal{H}_N^n\otimes\mu$ by the conjugation map $\Phi_N$, and the set $\Phi_N^{-1}(\Gamma_R(\tau,\delta,L,N))$
is the union over matrix sets~: $$\cup_{\Theta\in { U}\Gamma_R(\tau,\delta,L,N)}\Gamma_{orb}(\mathbf{X_1},...,\mathbf{X_n}:\Theta_1,...\Theta_n,N,L,\delta){\times\{\Theta\}}.$$ 
It follows that~:
\begin{align}\label{subadd2}
\begin{split}U\mu(\Gamma_R(\tau,\delta,L,N))&=
\int_{{ U}\Gamma_R(\tau,\delta,L,N)}\gamma_{N,\Theta,\delta,L}d\mu(\Theta)\\
&\leq{ \int_{{ U}\Gamma_R(\tau,\delta,L,N)}\gamma_{N,\Theta,\epsilon/2,K}d\mu(\Theta)\quad\text{by}\ \delta\leq\epsilon/2,\ L\geq K} \\
&\leq\gamma_{N,\Xi(N),\epsilon,K}\quad\text{by (\ref{thetaxi}) {and (\ref{invtheta})}.}
\end{split}\end{align}
Then combining (\ref{subadd1}), (\ref{subadd2}) and taking limits
 yields the inequality~:
$$\tchi_{orb}\leq (1-\eta)\chi_{orb}.$$%-f(\eta)$$
Since $\eta$ is arbitrary, we are done.

\medskip

{\sl $(7)$} The proof is a variant of the one in (6). First take some familiy $\Xi$  of hermitian matricies with $\tau_\Xi\in
 V_{\epsilon,K}(\tau_{\mathbf    {X_1},...,\mathbf{X_n}})$. Replacing $\Xi(N)$  by $\Xi$ in the arguments of the first part of (6) we deduce : $$\log \gamma_{N,\Xi,\epsilon,K}=\text{Ent}(\nu_{ \Xi}|\mathcal {H}_N^n)\leq \text{Ent}(\Psi_{ \Xi*}(\nu_{ \Xi})|U\Psi_{ \Xi*}(\nu_{ \Xi})).$$

Since $\tau_{\Psi_{ \Xi*}(\nu_{ \Xi})}\in V_{\epsilon,K}(\tau_{\mathbf    {X_1},...,\mathbf{X_n}})$ we obtain the following inequality~: $$\log \gamma_{N,\Xi,\epsilon,K}\leq \sup_{{\mu{ \in P}(H_N^R)^{\bar n}\atop \,\tau_\mu\in V_{\epsilon,K}(\tau_{\mathbf    {X_1},...,\mathbf{X_n}})}}
\text{Ent}( \mu |U \mu).$$
This implies the lower bound in the statement.

Assume now that 
$\tau=\tau_{\mathbf    {X_1},...,\mathbf{X_n}}$ extremal. Fix   $\eta,\epsilon,K$ choose $\delta,L$ as in lemma \ref{concentration}. For any $\mu\in P(H_N^R)^{\bar n}$ with  $\tau_\mu\in V_{\delta,L}(\tau)$ {\gree we have}~:
$$\text{Ent}(\mu|U\mu)\leq (1-\eta)\log \left[U\mu(\Gamma_R(\tau,\epsilon,K,N))\right]-f(\eta).
$$

With the same computation as in the proof of (\ref{subadd2}) we get the  inequality : 
$$U\mu(\Gamma_R(\tau,\epsilon,K,N))\leq \sup_{\Xi,\tau_\Xi\in V_{\epsilon,K}(\tau_{\mathbf    {X_1},...,\mathbf{X_n}}) }  \gamma_{N,\Xi,\epsilon,K}.$$
The second inequality of the statement follows.

\medskip

{\sl $(8)$} Let $\mathbf{X_{i}}=\{X_{i1},\ldots,X_{iP_i}\}$ and 
$\mathbf{Y_{i}}=\{Y_{i1},\ldots,Y_{iQ_i}\}$, with $\bar m=\sum_i Q_i$.
By   Kaplansky density theorem, for each $i$, one can find a set of non-commutative polynomials $P_{ij}(\mathbf{X_i}),\ j=1,\ldots,Q_i$  as close as we want in distribution to the set $\mathbf{Y_i}$.
{For such a family, we write : $$\mathbf{P_{i}}(\mathbf{X_i})=(P_{i1}(\mathbf{X_i}),...,P_{iQ_i}(\mathbf{X_i})),$$ $$\mathbf{P}(\mathbf{X_1},...,\mathbf{X_n})=(\mathbf{P_{1}}(\mathbf{X_1}),...,\mathbf{P_{n}}(\mathbf{X_n})).$$
}

{Let $\epsilon,K>0$.
One can find polynomials $P_{ij}(\mathbf{X_i}),\ j=1,\ldots,Q_i$, a real $\delta>0$ sufficiently small and an integer $L$ sufficiently large such that 
%(i)
 for all
 $\mu $ probability  measure on $(H_N^R)^{\bar n}$ in $V_{\delta,L}(\tau_{\mathbf{X_1},\ldots,\mathbf{X_n}})$
{\gree we have} $\mathbf{P}_{\star}\mu\in V_{\epsilon,K}(\tau_{\mathbf{Y_1},\ldots,\mathbf{Y_n}})$

%(ii)
%there exists  a $\mu\in V_{\delta,L}(\tau_{\mathbf{X_1},\ldots,\mathbf{X_n}})$
%with $\text{Ent}(\mu|U\mu)\geq \tchi_{orb}(\mathbf{X_1},\ldots,\mathbf{X_n})-\alpha$.

{Since $\Phi_n((U_1,...,U_n),(\mathbf{P_{1}}(\mathbf{X_1}),...,\mathbf{P_{n}}(\mathbf{X_n}))=(\mathbf{P_{1}}(U_1\mathbf{X_1}U_1^*),...\mathbf{P_{n}}(U_n\mathbf{X_n}U_n^*)),$} it is clear that $U\mathbf{P}_{\star}\mu=\mathbf{P}_{\star}U\mu$.
By (\ref{push}) {\gree we have}
$\text{Ent}(\mathbf{P}_{\star}\mu|U\mathbf{P}_{\star}\mu)\geq \text{Ent}(\mu|U\mu)$, therefore
$$\frac{1}{N^2}\sup_{\nu,\tau_\nu\in V_{\epsilon,K}(\tau_{\mathbf{Y_1},\ldots,\mathbf{Y_n}})}\text{Ent}(\nu|U\nu)\geq {\frac{1}{N^2} \sup_{\mu,\tau_\mu\in V_{\delta,L}(\tau_{\mathbf{X_1},\ldots,\mathbf{X_n}})}}\ \text{Ent}(\mu|U\mu).$$
{ Now  take a $\limsup$  then infimum over, successively, $\delta,L,\epsilon,K,$ to get the result}.%Since this is true for any $\epsilon,K,\alpha$, we get the result.
 }

\medskip

{\sl $(9)$}
First, choose a subsequence $N_m$ and $\mu_m $  probability measures on
$(H_{N_m}^R)^{\bar n}$ such that { $(\tau_{\mu_m })$ converges weakly to $\tau$ and :}
\begin{align*}\tchi_{orb}&(\mathbf{X_1};...;\mathbf{X_n})=\lim_{m\to \infty}\left(\frac{1}{N_m^2} \text{Ent}( \mu_m |U \mu_m)\right).\end{align*}
Without loss of generality, we assume that this orbital entropy is finite.
We follow arguments close to the proof  of  Lemma 3.4 in \cite{HiaiMUeda}. 
Remark that in the definition of $U\mu$ we can replace the unitary group $U(N)$ by $SU(N)$, since $U(N)$ acts by conjugation.
 Let $f_m(\mathbf{M})$ be the density of $\mu_m$ with respect to $U\mu_m$ (which exists if $m$ is sufficiently large).
Then for almost all values of $\mathbf{M}$, the function
 $$g_m(U_1,\ldots U_n,\mathbf{M})=f_m(U_1\mathbf{M}_1U_1^*,\dots,U_n\mathbf{M}_nU_n^*)$$ is a probability density in the variables 
$U_1,\ldots U_n$, with respect to the Haar measure $\mathcal{SH}_{N_{ m}}^n$ on $SU(N_m)^n$. 
For $\mathbf{M}$ let $\pi_{\mathbf{M}}$, be a probability measure on $SU(N_m)^{n}\times SU(N_m)^{n} $,  which is  an optimal coupling between $g_m(U_1,\ldots U_n,\mathbf{M})\mathcal{SH}_{N_{ m}}^n$ and $\mathcal{SH}_{N_{m}}^n$ for the { geodesic distance} on ${SU(N_m)}^{ n}$. This means that the marginals of the measure $\pi_{\mathbf{M}}$ on the two components of  
$SU(N_m)^{n}\times SU(N_m)^{n} $ are the measures $g_m(U_1,\ldots U_n,\mathbf{M})\mathcal{SH}_{N_{m}}^n$ and $\mathcal{SH}_{N_{m}}^n$, and the squared Wasserstein distance between the measures
$g_m(U_1,\ldots U_n,\mathbf{M})\mathcal{SH}_{N_{ m}}^n$ and $\mathcal{SH}_{N_{ m}}^n$ is
$$\int_{SU(N_m)^{n}\times SU(N_m)^{n}}{\left[d_{\text{geod}}((U_1,...,U_n),(V_1,...,V_n))\right]^2} %f_m(U_1,\ldots U_n,\mathbf{M})
d\pi_{\mathbf{M}}(U,V)$$
Such a measure can be constructed  measurably with respect to  $\mathbf{M}$ (see e.g. corollary 5.22 in \cite{Vill}).
We thus deduce { an estimate for the non-commutative 2-Wasserstein distance :}
\begin{align*}{ d_{W}}&{(\tau_{\mu_m},\tau_{U \mu_m})^2}\\&\leq \int dU \mu_m(\mathbf{M})\qquad \int d\pi_{\mathbf{M}}(\mathbf{U},\mathbf{V})\sum_{i=1}^{n}\sum_{j=1}^{P_i} \frac{1}{N_m}||U_i\mathbf{M}_{ij}U_i^*-V_i\mathbf{M}_{ij}V_i^*||^2_{HS}\\ &\leq \int dU \mu_m(\mathbf{M}) \int d\pi_{\mathbf{M}}(\mathbf{U},\mathbf{V})4R^2\tilde{ P}\frac{1}{N_m} \sum_{i=1}^{n}||\mathbf{U}_i-\mathbf{V}_i||_{HS}^2\\ &{\leq \int dU \mu_m(\mathbf{M}) \int d\pi_{\mathbf{M}}(\mathbf{U},\mathbf{V})4R^2\tilde{ P}\frac{1}{N_m} \left[d_{\text{geod}}((U_1,...,U_n),(V_1,...,V_n))\right]^2,}\end{align*}

{where we used the fact that the Hilbert-Schmidt distance can be majorized by the geodesic distance.}
Now using the Talagrand inequality of \cite{OttoVillani} on $SU(N_m)^{n}$,  as in Proposition 3.5 of \cite{HiaiMUeda} {
 } {\gree we get} :
\begin{eqnarray*}
{d_{W}(\tau_{\mu_m},\tau_{U \mu_m})^2}&\leq& (4R)^2\tilde{ P}\times
\\&&\qquad\int dU \mu_m(\mathbf{M}) \frac{-1}{N_m^2} \text{Ent}(g(U_1,\ldots,U_n,\mathbf{M})\mathcal{SH}_{N_{ m}}^n(U)|\mathcal{SH}_{N_{ m}}^n(U))
\end{eqnarray*}

{
Now we can use the fact that $U\mu_m$ is invariant by the action of $U(N)^n$ and interchange the order of integration to get

\begin{align*}
&\int \text{Ent}(g(U_1,\ldots,U_n,\mathbf{M})\mathcal{SH}_{N_{ m}}^n(U)|\mathcal{SH}_{N_{m}}^n(U))dU \mu_m(\mathbf{M}) 
 \\
&=\int \int f_m(U_1\mathbf{M}_1U_1^*,\dots,U_n\mathbf{M}_nU_n^*)\log f_m(U_1\mathbf{M}_1U_1^*,\dots,U_n\mathbf{M}_nU_n^*)d\mathcal{H}_{N_{ m}}^n(\mathbf{U})
 dU\mu_m(\mathbf{M})  
 \\&
 =\text{Ent}(\mu_m|U\mu_m), \end{align*}
thus
$${ d_{W}(\tau_{\mu_m},\tau_{U \mu_m})^2}\leq\frac{-(4R)^2\tilde P}{N_m^2} \text{Ent}(\mu_m|U\mu_m).$$}
{ By our choice of $\mu_m$,} the noncommutative distribution of the random matrix sets $\mathbf{M}$ under
 $\mu_m$ converges weakly to $\tau$ as $N\to\infty$. Let us check that similarly, under 
$U\mu_m$ this noncommutative distribution converges  weakly to $\tau_{free}$.
This is a consequence of Remark 3.2 in \cite{Collins02}. Indeed, there it is proved that 
 $\mathbf{M}$ is asymptotically free from $\{U_1\}$,..., $\{U_n\}$ (independant Haar unitaries)  provided the distribution of $\mathbf{M}$ concentrates around its mean. But this concentration is provided by Lemma \ref{concentration}.
We leave the easy but tedious details to the reader. 

{ As a consequence of Talagrand's inequality, the only if part of the characterization of freeness is obvious. Now assume $\tau=\tau_{free}$ and take $\mu_m$ as above so that now $\tau_{U\mu_m}$ tends weakly to  $\tau=\tau_{free}$. Thus for $m$ large enough so that $\tau_{U\mu_m}$ $\epsilon,K$ approximates $\tau$ we have  $$0=\left(\frac{1}{N_m^2} \text{Ent}( U\mu_m |U \mu_m)\right)\leq \frac{1}{N_m^2}\sup_{{\mu \in P(H_{N_m}^R)^{\bar n}\atop \,\tau_\mu\in
 V_{\epsilon,K}(\tau_{\mathbf    {X_1},...,\mathbf{X_n}})} } \text{Ent}( \mu |U \mu).$$
 
 As a consequence taking a { limit in $m$} and then in $\epsilon,K$ since they are arbitrary in the argument above, {\gree we get} $\tchi_{orb}(\mathbf{X_{1}},\dots,\mathbf{X_{n}})=0$.}

\end{proof} 
 
 \section{Preliminaries about entropy, marginals and unitary invariant versions of a measure}
Before giving several other generalizations of orbital entropy, we start with some preliminary results.

Let $\mu\in P((H_R^N)^{\bar{n}})$, considered as the probability distribution of a  family of random matrices  $(\mathbf{A_1},...,\mathbf{A_n})$, where each  $\mathbf{A_i}$ consists in a bunch of variables like $\mathbf{X_i}$. Again $U\mu$ is then the law of $(U_1\mathbf{A_1}U_1^*,...,U_n\mathbf{A_n}U_n^*)$ where $U_i$ are independent  variables distributed with respect to the Haar measure ${\mathcal H}_N$ of the unitary group $U(N)$.
More generally, we consider partial conjugations in the following way~:
if  $\pi:[1,n]\to [1,\ell]$ is a surjective map (equivalently, we can consider the partition  $\Pi=\{\pi^{-1}(i)\}_{i=1,\ldots,\ell}$ of $[1,n]$ which it defines) we denote $U^{\pi}\mu(=U^\Pi\mu)$ the law of $(U_{\pi(1)}\mathbf{A_1}U_{\pi(1)}^*,...,U_{\pi(n)}\mathbf{A_n}U_{\pi(n)}^*)$
where the $U_i,i=1\ldots,\ell$ are independent Haar random unitary matrices.
We will write $U^G$ for the global unitary invariant version, corresponding to  $\Pi=\{\{1,...,n\}\}$.
It is clear that for any absolutely continuous measure $\mu$ the measure
$U^{\pi}\mu$ is absolutely continuous. 

\begin{lemma}\label{UEnt}
\begin{enumerate}
\item[(i)]
 Let $\mu,\nu\in P((H_R^N)^{\tilde{n}})$ with $U^{\pi}\nu=\nu$, then $$Ent(\mu|\nu)=Ent(\mu|U^{\pi}\mu)+Ent(U^{\pi}\mu|\nu).$$
\item[(ii)] Let $\mu\in P((H_R^N)^{\bar{n}})$ and $\nu=\bigotimes_i\nu_i$, $\nu_i\in P((H_R^N)^{\tilde{P_i}})$. We denote   $q_1\mu$ and $\ q_2\mu$ the marginals for the bunch of variables corresponding, respectively,  to $(\mathbf{A_1},...,\mathbf{A_m})$ and $(\mathbf{A_{m+1}},...,\mathbf{A_n})$, then $$\text{Ent}(\mu|\nu)=\text{Ent}(\mu|q_1\mu\otimes q_2\mu)+\text{Ent}(q_1\mu\otimes q_2\mu|\nu).$$
\item[(iii)]With the notations of (ii) and $V=U^\Pi$ for $\Pi=\{\{1,..,m\},\{m+1,..,n\}\}$  we have :
$$\text{Ent}(V\mu|U\mu)\leq  \text{Ent}(q_1U^G\mu|Uq_1\mu)+\text{Ent}(q_2U^G\mu|Uq_2\mu).$$
\end{enumerate}
\end{lemma}
\begin{proof}

{\sl $(i)$}
This is a generalization to relative entropy of an equality in the proof of Theorem \ref{prop} (5) above.
Without loss of generality we assume $\mu\ll\nu$ {since if we don't have both  $\mu\ll U^{\pi}\mu$ and $ U^{\pi}\mu\ll \nu$, the right hand side is $-\infty$ and the equality is true if we don't have $\mu\ll\nu$, so that in any case we can assume $\mu\ll\nu$.} Consider $\rho=\frac{d\mu}{d\nu}$. Since $\nu$ 
is unitarily invariant, we have $U^{\pi}\mu\ll U^{\pi}\nu=\nu$. Moreover, 
\begin{eqnarray*}
\rho_U(\mathbf{A_1},...,\mathbf{A_n})&:=&
\frac{dU^{\pi}\mu}{d\nu}(\mathbf{A_1},...,\mathbf{A_n}) \\
&=&\int d{\mathcal H}_N^{ \ell}(U_1,...,U_{ \ell})\rho(U_{\pi(1)}\mathbf{A_1}U_{\pi(1)}^*,...,U_{\pi(n)}\mathbf{A_n}U_{\pi(n)}^*).\end{eqnarray*}

Using (\ref{push}) we have $\text{Ent}(\mu|\nu)\leq \text{Ent}(U^{ \pi}\mu|\nu)$ and we can thus compute : \begin{align*}\text{Ent}(\mu|\nu)&=-\int \rho \ln(\rho) d\nu=-\int \rho \ln(\rho_U)d\nu-\int \rho \ln(\frac{\rho}{\rho_U})d\nu \\&= -\int \rho_U \ln(\rho_U)d\nu+\text{Ent}(\mu|U\mu)\\ &=\text{Ent}(U\mu|\nu)+\text{Ent}(\mu|U\mu),\end{align*}
where, in the third line, we used  unitary invariance to replace $\rho$ by $\rho_U$. The reverse implication starting from finiteness of the left hand side is also clear.

\medskip

{\sl $(ii)$} The proof is similar to {\sl (i)}. In order  to solve finiteness issues, one can again use (\ref{push}) to get $\text{Ent}(\mu|\nu)\leq\text{Ent}(q_i\mu| q_i\nu)$.

\medskip

{\sl $(iii)$}
 The inequality comes from subadditivity of entropy. Indeed
 consider, without loss of generality, $\rho_V$
  the density of $V\mu$ with respect to $U\mu$.
Using unitary invariance of $U\mu$ %\nu^{\otimes}$ and $\rho_U$
 we get :
\begin{align*}\text{Ent}&(V\mu|U\mu)=-\int \rho_V\ln(\rho_V)dU\mu\\&
=-\int dU\mu(\mathbf{A}) \int R(U_1,...,U_n,\mathbf{A})\ln(R(U_1,...,U_n,\mathbf{A}))d{\mathcal H}_N^n(U_1,...,U_n),\end{align*}
where, for a.e. $\mathbf{A}=(\mathbf{A_1},...,\mathbf{A_n})$, the quantity $$R(U_1,...,U_n,\mathbf{A})=\rho_{V}(U_1\mathbf{A_1}U_1^*,...,U_n\mathbf{A_n}U_n^*)$$ is  a probability density on 
$U(N)^n$.  Let $R_1,R_2$ be the densities of marginals, namely, with obvious notations, $$R_1(\mathbf{U_2},\mathbf{A})=\int d{\mathcal H}_N^m(\mathbf{U_1})R(\mathbf{U_1},\mathbf{U_2},\mathbf{A}),$$ $$R_2(\mathbf{U_1},\mathbf{A})=\int d{\mathcal H}_N^{n-m}(\mathbf{U_2})R(\mathbf{U_1},\mathbf{U_2},\mathbf{A}).$$
{\gree we have} $$R_2(\mathbf{U_1},\mathbf{A})=R_2(\mathbf{I},U_1\mathbf{A_1}U_1^*,\ldots, U_m\mathbf{A_m}U_m^*,\mathbf{A_{m+1},\ldots,A_n}).$$ Moreover
 $$R_2(\mathbf{I},\mathbf{A})dU\mu(\mathbf{A})=dV\mu(\mathbf{A})$$ is a probability measure with marginal $q_1V\mu=q_1U^G\mu.$ 
Using the subadditivity of ordinary entropy relative to a product measure, we get~:
\begin{align*}\text{Ent}(V\mu|U\mu)&\leq-\int dU\mu(\mathbf{A})R_2(\mathbf{I},\mathbf{A})\ln (R_2(\mathbf{I},\mathbf{A}))
\\&\qquad -\int dU\mu(\mathbf{A})R_1(\mathbf{I},\mathbf{A})\ln (R_1(\mathbf{I},\mathbf{A}))
\\&\leq \text{Ent}(q_1U^G\mu|Uq_1\mu)+\text{Ent}(q_2U^G\mu|Uq_2\mu).
\end{align*}

\end{proof}

\section{Variants and extensions}
\subsection{Overview}
The main drawback of our definition of orbital free entropy is that
we are unable to prove equality in part (5) of Theorem \ref{prop}.
In order to overcome this problem,
as mentionned at the beginning of section 7, several other generalizations of orbital entropy may be considered.
We will describe below two variants which we call maximal mutual entropy and I-mutual entropy. The last one satisfies the required additivity property however we lose the fact that it depends only on the subalgebras generated by the subset of variables.
Let us describe briefly the content of this section.
First, we can consider, as for Voiculescu's entropy, a variant of free entropy {\bl in the presence of another set of variables}, which plays a dummy role in the definition. This will be considered in section \ref{presence}.
Instead of using the relative entropy of $\mu$, an approximating measure, with respect to its unitary invariant mean $U\mu$, we can consider the relative entropy with respect to the product of the marginal distributions 
 of $U\mu$ with respect to the subsets. This yields a quantity which we call maximal mutual entropy, and which we consider in section \ref{maxmut}. Again this quantity depends only on the $W^*$ algebras generated by the subsets, and is subadditive.
Another alternative  is to use  Ciszar's I-projection first and then to take the relative entropy of this specific measure with respect to the tensor product of its marginals (which are automatically unitary invariant in this case). This gives  what we call I-mutual entropy, studied in section \ref{section_Imut}. 
This quantity satisfies  a strong  additivity property (property below), which generalizes the additivity of 
the orbital entropy {\bl of} \cite{HiaiMUeda}. Unfortunately, we are not able to prove
 that {\bl it} depends only on the $W^*$ algebras generated by the subsets.
 All these entropies coincide with  orbital entropy defined in \cite{HiaiMUeda} in the context they define it. 
It is plausible that they always coincide, although we do not have a proof of this fact at this stage.

\subsection{Orbital  entropy in the presence of other variables}\label{presence}
As in section 6, we consider finite sets of non-commutative random variables 
$\mathbf{X_i}=\{X_{i1},...,X_{iP_i}\}$ for $i=1,\ldots,n$, and $\bar{n}=\sum_i P_i$, while  $\mathbf{Y}=\{Y_1,\ldots,Y_t\}$ is likewise a multivariable containing $t$ variables.
Their joint non-commutative (tracial) distribution is  $\tau=\tau_{\mathbf{X_1};...;\mathbf{X_n};\mathbf{Y}}\in \mathcal S_R^{\bar{n}+t}$.
We will use the notation :$$A_{N,\epsilon,K}(\tau)=\{ \mu\in P((H_R^N)^{\bar{n}+t})\ | \ \tau_{\mu} \in V_{\epsilon,K}(\tau) \}.$$
Also we denote $p\mu$ the marginal distribution of $\mu$ on the $\mathbf X$ variables.
 \begin{definition}
The \textit{\textbf{free orbital entropy } of $\mathbf{X_1},...,\mathbf{X_n}$ \textbf{in the presence of} $\mathbf{Y}$ is, if $\tau=\tau_{\mathbf    {X_1},...,\mathbf{X_n},\mathbf{Y}}$ }:
     $$\tchi_{orb}(\mathbf{X_1};...;\mathbf{X_n}:\mathbf{Y})
                  =\sup_{ R\geq\mathcal{R}(\tau)} \lim_{K\to\infty,\epsilon\to 0}
                \limsup_{N\to\infty}
\left(\frac{1}{N^2}
\sup_{\mu \in A_{N,\epsilon,K}
(\tau)}
\text{Ent}( p\mu |Up\mu)\right).$$
\end{definition}

The orbital free entropy in the presence of {\bl other variables} satisfies  properties similar to the ones of Theorem \ref{prop}, the proofs being
easy variations on the proofs for the orbital free entropy. 
We state here only an improved version of the additivity property. 

\medskip
\begin{theorem}\label{chiorbep}
\begin{align*}&\tchi_{orb}(\mathbf{X_1};...;\mathbf{X_m};\mathbf{X_{m+1}};...;\mathbf{X_{n}}:\mathbf{Y})\leq  {\bl \tchi_{orb}(\mathbf{X_1}\cup...\cup\mathbf{X_m};\mathbf{X_{m+1}}\cup...\cup\mathbf{X_{n}}:\mathbf{Y})+}\\ &\tchi_{orb}(\mathbf{X_1};...;\mathbf{X_m}:\mathbf{X_{m+1}}\cup...\cup\mathbf{X_{n}}\cup\mathbf{Y})+\tchi_{orb}(\mathbf{X_{m+1}};...;\mathbf{X_n}:\mathbf{X_{1}}\cup...\cup\mathbf{X_{m}}\cup\mathbf{Y}).\end{align*}
\end{theorem}
\begin{proof}
Write for $\mu$ in $ A_{N,\epsilon,K}(\tau)$ $Vp\mu$ as in lemma \ref{UEnt} (iii) the unitary invariant variant for blocks. Note that $pU^G\mu=U^Gp\mu$ and~: $$Ent(p\mu|Up\mu)=Ent(p\mu|U^Gp\mu)+ Ent(U^Gp\mu|Up\mu)\leq Ent(U^Gp\mu|Up\mu),$$ (from lemma \ref{UEnt} (i)) so that, since $U^G\mu\in A_{N,\epsilon,K}(\tau)$, we may assume $\mu=U^G\mu$ when we bound orbital entropy.
%, i.e. if $p\mu$ is the law of  $(\mathbf{A_1},...,\mathbf{A_m})$, $Vp\mu$ is then the law of $(V_1\mathbf{A_1}V_1^*,...,V_1\mathbf{A_n}V_1^*,V_2\mathbf{A_{n+1}}V_2^*,...,V_2\mathbf{A_m}V_2^*)$ where $V_i$ are independent  variables distributed with respect to the Haar measure of the unitary group $\gamma_{U(N)}$. Without loss of generality $\mu$ absolutely continuous with respect to $U\mu$ (and thus to $V\mu$). Then
Applying lemma \ref{UEnt} (i) and (iii) we get the concluding estimate  \begin{align*}\text{Ent}(p\mu|Up\mu)&=\text{Ent}(p\mu|Vp\mu)+\text{Ent}(Vp\mu|Up\mu)\\&\leq \text{Ent}(p\mu|Vp\mu)+ \text{Ent}(q_1\mu|Uq_1\mu)+\text{Ent}(q_2\mu|Uq_2\mu).\end{align*}
\end{proof}
\subsection{Maximal mutual entropy}\label{maxmut}
We use the same notations as in the preceding section, and denote $p_1,\ldots,p_n$ the projections on the sets of variables 
$\mathbf{X_1};...;\mathbf{X_n}$.

  \begin{definition}
     The \textbf{\textit{free maximal mutual entropy of $\mathbf{X_1},...,\mathbf{X_n}$ in the presence of $\mathbf{Y}$}} is , if $\tau=\tau_{\mathbf{X_1},...,\mathbf{X_n},\mathbf{Y}}$~:
          \begin{align*}&\tchi_{Mmut}(\mathbf{X_1};...;\mathbf{X_n}:\mathbf{Y})=\\
&\sup_{ R\geq\mathcal{R}(\tau)}\lim_{K\to\infty,\epsilon\to0}
    \limsup_{N\to\infty} \left(\frac{1}{N^2}\sup_{\mu\in A_{N,\epsilon,K}(\tau)} \text{Ent}(p \mu |p_1Up \mu\otimes... \otimes p_nUp \mu \right)\end{align*}
    If $\mathbf{Y}$ is empty we just write 
$\tchi_{Mmut}(\mathbf{X_1};...;\mathbf{X_n})$.
   \end{definition}
   Note that the limits in $\epsilon,K$ are actually infima.
 { 
We also define a notion of relative entropy to state the best subadditivity result. We compare it in the next subsection, but note already that it coincides with the definition of section 4 when $\mathbf{Y}=\emptyset$.

  \begin{definition}  
     We  define,  for $\tau=\tau_{\mathbf    {X_1},...,\mathbf{X_n},\mathbf{Y}}$, a \textit{\textbf{random microstate free entropy in the presence}} of $\mathbf{Y}$ as :
     \begin{align*}
\tchi&(\mathbf{X_1};...;\mathbf{X_n}:\mathbf{Y})=
\\
&\sup_{ R\geq\mathcal{R}(\tau)}\lim_{K\to\infty,\epsilon\to0}
     \limsup_{N\to\infty}\left(\frac{1}{N^2}\sup_{\mu\in A_{N,\epsilon,K}(\tau)} \text{Ent}(p \mu )+\frac{n}{2}\log{N} \right)\end{align*}
   \end{definition}
  }
\begin{theorem}\label{Mmut}The free maximal mutual entropy satisfies the following properties :
\begin{enumerate}
\item (Vanishing for one variable) 
$$\tchi_{Mmut}(\mathbf{X_1})=0,$$
 for any single multivariable {\bl having finite-dimensional approximants}.

\item (Improved Subadditivity) \begin{align*}&\tchi_{Mmut}(\mathbf{X_1};...;\mathbf{X_m};\mathbf{X_{m+1}};...;\mathbf{X_{n}}:\mathbf{Y})\leq {\bl \tchi_{Mmut}(\mathbf{X_1}\cup...\cup\mathbf{X_m};\mathbf{X_{m+1}}\cup...\cup\mathbf{X_{n}}:\mathbf{Y})+}\\& \tchi_{Mmut}(\mathbf{X_1};...;\mathbf{X_m}:\mathbf{X_{m+1}}\cup...\cup\mathbf{X_{n}}\cup\mathbf{Y})+\tchi_{Mmut}(\mathbf{X_{m+1}};...;\mathbf{X_n}:\mathbf{X_{1}}\cup...\cup\mathbf{X_{m}}\cup\mathbf{Y}).
\end{align*}
\item (Improved subadditivity of entropy) \begin{align*}
\tchi&(\mathbf{X_1},\mathbf{X_{2}}{ :\mathbf{Y}})\leq \tchi_{Mmut}(\mathbf{X_1};\mathbf{X_{2}}{:\mathbf{Y}})+\tchi(\mathbf{X_1}:\mathbf{X_{2}}{\cup\mathbf{Y}})+\tchi(\mathbf{X_2}:\mathbf{X_{1}}{\cup\mathbf{Y}}).
 \end{align*}
 \item (Agreement with previous definition) 
If  $\mathbf{X_i}$ are hyperfinite multivariables then 
$$\chi_{orb}(\mathbf{X_{1}};...;\mathbf{X_n})\leq \tchi_{Mmut}(\mathbf{X_{1}};...;\mathbf{X_n})$$
(where free orbital entropy is in the sense of \cite{HiaiMUeda}).
If moreover  $\tau_{\mathbf{X_{1}},...,\mathbf{X_{n}}}$ is extremal then
 $$\tchi_{Mmut}(\mathbf{X_{1}};...;\mathbf{X_n})=\chi_{orb}(\mathbf{X_{1}};...;\mathbf{X_n}).$$ 
 \item (Dependence on algebras)
 If $\mathbf{X_{1}},\dots,\mathbf{X_{n}}$, $\mathbf{Y_{1}},\dots,\mathbf{Y_{n}}$ are  multi-variables such that
$\mathbf{Y_{i}}\subset W^*(\mathbf{X_{i}})$ for $1\le i\le n$, then
$$
\tchi_{Mmut}(\mathbf{X_{1}},\dots,\mathbf{X_{n}})\leq\tchi_{Mmut}(\mathbf{Y_{1}},\dots,\mathbf{Y_{n}}).
$$
In particular, $\tchi_\mathrm{Mmut}(\mathbf{X_{1}},\dots,\mathbf{X_{n}})$ 
 depends only upon
$W^*(\mathbf{X_{1}}),\dots,W^*(\mathbf{X_{n}}).$
 \end{enumerate}
 \end{theorem}
 
\begin{proof} The proofs of
{\sl $(1),(5)$} {are} similar to the corresponding properties of $\tchi_{orb}$.

\medskip

(2)
Let $p$ be the projection on the $\mathbf{X}$ variables, 
$p_i$ the projection $\mathbf{X_i}$, and 
$q_1,q_2$  the projections on $\mathbf{X_1},\ldots,\mathbf{X_m}$
and $\mathbf{X_{m+1}},\ldots,\mathbf{X_n}$, respectively. Let $\mu\in  A_{N,\epsilon,K}(\tau)$,  
 we may assume, {as in the proof of Theorem \ref{chiorbep},}  $\mu=U^G\mu$, so that  we have $p_i Up\mu=p_i \mu$ and $q_iV\mu=q_i\mu$.
Applying lemma \ref{UEnt} (ii) we get~:
\begin{align*}\text{Ent}(p \mu|\bigotimes_i p_i Up\mu)=\text{Ent}(p \mu|q_1Vp\mu\otimes q_2Vp\mu)+\text{Ent}(q_1\mu\otimes q_2\mu|\bigotimes_i p_iUp\mu)\end{align*}
%The first term is $\leq 0$, 
And we have~:
$$\text{Ent}(q_1\mu\otimes q_2\mu|\bigotimes_i p_iUp\mu)=
\text{Ent}(q_1\mu|\bigotimes_{i=1,\ldots,m} p_iUp\mu)+
\text{Ent}(q_2\mu|\bigotimes_{i=m+1,\ldots,n} p_iUp\mu).
$$
Taking {suprema and} limits yields the inequality.

\medskip

 (3)
With a similar notation as in the previous point, we take $\mu=U^G\mu$ in $ A_{N,\epsilon,K}(\tau)$, then~:
\begin{align*}Ent&(p \mu | Leb)=\text{Ent}(p \mu |q_1V\mu\otimes q_2V\mu)+\text{Ent}(q_1\mu\otimes q_2\mu|Leb)\end{align*}
and again we may take { suprema and} limits to get the required conclusion.

\medskip

(4) This follows from Theorem \ref{prop} (6), as well as Theorem \ref{Imut} (4) and Proposition \ref{rel} to be proved below. 
 
\end{proof}

\subsection{I-mutual entropy}\label{section_Imut}
In order to extend { again in this subsection} \cite{HiaiMUeda} for (not necessarily hyperfinite) multivariables, we consider multivariables $\mathbf{X_i}=(\mathbf{X_{i1}},...,\mathbf{X_{iP_i}})$  where each
$\mathbf{X_{ij}}$ is itself a family  of  hyperfinite multivariables, i.e. 
 $\mathbf{X_{ij}}=\{X_{ij1},...,X_{ijQ_{ij}}\}$ and  $\tilde{P_i}=\sum_{j=1}^{P_i}Q_{ij}$,
$\bar{n}=\sum_{i=1}^n\tilde{P_i}$. 
 For the definition of free entropy in presence we consider  also  analogously $\mathbf{Y}=(\mathbf{Y_{1}},...,\mathbf{Y_{P}})$ containing $\bar{t}$ variables.
For technical reasons (in order to get values agreeing with those of \cite{HiaiMUeda} in the hyperfinite case) we will let the approximations  depend on doubled parameters 
$\epsilon=(\epsilon_1,\epsilon_2),K=(K_1,K_2)$.

We first consider $\sigma_{ij}=\sigma_{N,\epsilon_2,K_2}(\tau_{\mathbf{X_{ij}}})$ the normalized restriction of  Lebesgue measure to the  set $V_{\epsilon_2,K_2}(p_{ij}\tau)$ of $\epsilon_2,K_2$ approximations of $p_{ij}\tau=\tau_{\mathbf{X_{ij}}}$ where $p_{ij}$ gives the marginals on the $ij$-th bunch of hyperfinite variables.
We denote by $C_{N,\epsilon,K}(\tau)$   
 Csiszar's I-projection of $S_{N,\epsilon_2,K_2}(\tau):=\bigotimes_{i,j}\sigma_{ij}$
 on~: $$A_{N,\epsilon,K}(\tau)=\{ \mu\in P((H_R^N)^{\bar{n}+{\bar{t}}})\ | \ \tau_{\mu} \in V_{\epsilon_1,K_1}(\tau)\ \forall i,j
 \  p_{ij} \tau_{\mu}\in V_{\epsilon_2,K_2}(p_{ij}\tau)\}.$$  Thus, we allow us to approximate better the hyperfinite marginals. 
 This will be used to define an I-mutual entropy with good additivity properties, which was a motivation for Voiculescu's non-microstates mutual information and for Hiai-Miyamoto-Ueda's microstate variant. However the other  variants seem to be better behaved in every other respects. We will use not only a free ultrafilter $\omega$ on the integers but also a point $\theta$ in the boundary of the Stone-{\bl \v{C}ech} compactification of $(0,1]$. 
If $A_{N,\epsilon,K}(\tau)$ does not contain elements of finite entropy, any entropy involving $C_{N,\epsilon,K}$ (thus undefined) is by convention $-\infty$.
Likewise,  a sup over an empty set is $-\infty$.

  \begin{definition}
Let $\tau=\tau_{\mathbf{X_1},...,\mathbf{X_n},\mathbf{Y}}$,
   we define \textit{\textbf{I-mutual entropy}} as 
\begin{align*}&\tchi_{Imut}(\mathbf{X_1};...;\mathbf{X_n}:\mathbf{Y})=\sup_{ R\geq \mathcal{R}(\tau)}\limsup_{\epsilon_1\to0}\limsup_{K_1\to\infty}\limsup_{\epsilon_2\to0}\limsup_{K_2\to\infty}\limsup_{N\to\infty}\\ &\left(\frac{1}{N^2}\text{Ent}(p C_{N,\epsilon,K}(\tau)|p_1UpC_{N,\epsilon,K}(\tau)\otimes... \otimes p_nUp C_{N,\epsilon,K}(\tau)\right),\end{align*}
  where $p_i$ is the projection on submultivariables  $\mathbf{X_i}$ and $p$ on $\mathbf{X_1},...,\mathbf{X_n}$.   We  write $\tchi_{Imut}(\mathbf{X_1};...;\mathbf{X_n})$ when $\mathbf{Y}=\emptyset$. Likewise we define  $\tuchi_{Imut}(\mathbf{X_1};...;\mathbf{X_n}:\mathbf{Y})$ a liminf variant (with respect to $N,\epsilon,K$) of I-mutual entropy  and an ultrafilter variant $\totchi_{Imut}(\mathbf{X_1};...;\mathbf{X_n}:\mathbf{Y})$ 
%  \noindent 
(with $\lim_{1/R\to\theta}\lim_{\epsilon_1\to\theta}\lim_{K_1\to\omega}\lim_{\epsilon_2\to\theta}\lim_{K_2\to\omega}\lim_{N\to\omega}$).% with a limit to \omega

   \end{definition}

      We will also need a notion of free I-entropy in the presence of other variables to get additivity properties with I-mutual entropy. Instead of maximizing the entropy of the projection of measures also approximating $\mathbf{Y}$, which would be more natural in the spirit of Voiculescu's definition and correspond to the definition taken in the previous subsection, we take Csiszar's projection including approximation of  $\mathbf{Y}$, we project and take entropy.
   
    \begin{definition}  We define \textit{\textbf{free I-entropy in the presence}} of $\mathbf{Y}$ as :
     \begin{align*}\tchi_{I}&(\mathbf{X_1};...;\mathbf{X_n}:\mathbf{Y})=\sup_{ R\geq \mathcal{R}(\tau_{\mathbf{X_1},...,\mathbf{X_n},\mathbf{Y}})}\limsup_{K_1\to\infty,\epsilon_1\to0}\limsup_{K_2\to\infty,\epsilon_2\to0}\limsup_{N\to\infty}\\ &\left(\frac{1}{N^2}\text{Ent}(p C_{N,\epsilon,K}(\tau_{\mathbf{X_1},...,\mathbf{X_n},\mathbf{Y}}))+\frac{n}{2}\log{N} \right),\end{align*}
     and likewise $\totchi_{I}(\mathbf{X_1};...;\mathbf{X_n}:\mathbf{Y})$% with a limit to \omega
     , $\tuchi_{I}(\mathbf{X_1};...;\mathbf{X_n}:\mathbf{Y})$.
     
%     We also define a \textit{\textbf{random microstate free entropy in the presence}} of $\mathbf{Y}$ as :
%     \begin{align*}
%\tchi&(\mathbf{X_1};...;\mathbf{X_n}:\mathbf{Y})=
%\\
%&\sup_R\lim_{K\to\infty,\epsilon\to0}
%     \limsup_{N\to\infty}\left(\frac{1}{N^2}\sup_{\mu\in A_{N,\epsilon,K}(\tau_{\mathbf{X_1},...,\mathbf{X_n},\mathbf{Y}})} \text{Ent}(p \mu )+\frac{n}{2}\log{N} \right)\end{align*}
   \end{definition}
  % We will write later $\tchi(\mathbf{X_1};...;\mathbf{X_n}:\mathbf{Y})=sup_R\tchi_R(\tau_{\mathbf{X_1},...,\mathbf{X_n},\mathbf{Y}},\bar{n},\bar{n})$ with the notation of Proposition \ref{ExtDef}.
%Of course when there is no $\mathbf{Y}$, {\gree we have}  $\tchi(\mathbf{X_1};...;\mathbf{X_n}:\mathbf{Y})=\tchi(\mathbf{X_1}\cup...\cup\mathbf{X_n})$ defined in section 4.
    We have inequalities, as in sections 4 and 6, given in the following lemma.
   
   \begin{lemma}
We have :
\begin{align*}{ \tchi_I(\mathbf{X_1},...,\mathbf{X_n}:\mathbf{Y})}&\leq \tchi(\mathbf{X_1},...,\mathbf{X_n}:\mathbf{Y}),\\ 
{
\chi(\mathbf{X_1};...;\mathbf{X_n}:\mathbf{Y})}&\leq \tchi(\mathbf{X_1},...,\mathbf{X_n}:\mathbf{Y}),
\\{
\chi(\mathbf{X_1};...;\mathbf{X_n})}&\leq \tchi_I(\mathbf{X_1},...,\mathbf{X_n}),
\end{align*}
 and corresponding ultrafilter, liminf variants.
 
  If
{
 $\tau_{\mathbf{X_1};...;\mathbf{X_n},\mathbf Y}$ is extremal we also have :
$$\chi(\mathbf{X_1},...,\mathbf{X_n}:\mathbf{Y})=  \tchi(\mathbf{X_1},...,\mathbf{X_n}:\mathbf{Y}).$$
Especially,  if $\tau_{\mathbf{X_1};...;\mathbf{X_n}}$ is extremal we have :$$\chi(\mathbf{X_1};...;\mathbf{X_n})= \tchi_{I}(\mathbf{X_1};...;\mathbf{X_n})=\tchi(\mathbf{X_1};...;\mathbf{X_n}).$$}
\end{lemma}

\begin{proof} Let  $\tau=\tau_{\mathbf{X_1};...;\mathbf{X_n}:\mathbf{Y}}$
Since $C_{N,\epsilon,K}\in A_{N,(\epsilon_1,\epsilon_1),(K_1,K_1)}(\tau)$ by  definition we obtain
$\tchi_{I}(\mathbf{X_1};...;\mathbf{X_n}:\mathbf{Y})\leq \tchi(\mathbf{X_1},...,\mathbf{X_n}:\mathbf{Y})$. 

 The inequalities between $\chi$ and $\tilde{\chi}$ are similar to those in sections 4 and 6. 
{Let us merely outline the proofs for the reader's convenience. 
First, recall Voiculescu's definition from \cite{Voi3} :
     \begin{align*}
\chi&(\mathbf{X_1};...;\mathbf{X_n}:\mathbf{Y})=
&\sup_{ R\geq\mathcal{R}(\tau)}\lim_{K\to\infty,\epsilon\to0}
     \limsup_{N\to\infty}\left(\frac{1}{N^2} \log (p \Gamma_R(\tau,\epsilon,K,N))+\frac{n}{2}\log{N} \right),\end{align*}
where $pA\in (H_R^N)^{\bar{n}}$ is now the projection of the set $A\in (H_R^N)^{\bar{n}+t}$.

Fix $\epsilon,K>0$.
For $\mathbf{M}\in p\Gamma_R(\tau,\epsilon,K,N)$, we consider the fiber : $$\Gamma_{R,\mathbf{M}}= (\{\mathbf{M}\}\times(H_R^N)^{t})\cap \Gamma_R(\tau,\epsilon,K,N).$$
We define a probability measure $\mu$ with support in $\Gamma_R(\tau,\epsilon,K,N)$ (so that $\tau_\mu\in V_{\epsilon,K}(\tau)$), on a measurable set $A\in (H_R^N)^{\bar{n}+t}$ by :
$$\mu(A)=\frac{1}{\text{Leb}(p\Gamma_R(\tau,\epsilon,K,N))}\int_{ p\Gamma_R(\tau,\epsilon,K,N)}d\text{Leb}_{(H_R^N)^{\bar{n}}}(\mathbf{M})\frac{(\delta_\mathbf{M}\times \text{Leb}_{(H_R^N)^t})(A\cap\Gamma_{R,\mathbf{M}})}{(\delta_\mathbf{M}\times \text{Leb}_{(H_R^N)^t})(\Gamma_{R,\mathbf{M}})}.$$ 
By definition, we get $p\mu(B) =\frac{1}{\text{Leb}(p\Gamma_R(\tau,\epsilon,K,N))}\text{Leb}(B\cap p\Gamma_R(\tau,\epsilon,K,N)))$, so that : $$\log (p \Gamma_R(\tau,\epsilon,K,N))=\text{Ent}(p\mu)\leq \sup_{\mu\in A_{N,\epsilon,K}(\tau)} \text{Ent}(p \mu ).$$
We conclude $\chi(\mathbf{X_1};...;\mathbf{X_n}:\mathbf{Y})\leq \tchi(\mathbf{X_1},...,\mathbf{X_n}:\mathbf{Y})$.

Conversely, assume $\tau$ extremal. Fix $\eta,\epsilon,K>0$ and choose $\delta,L$ as in Lemma \ref{concentration}  so that, if $\mu\in A_{N,\delta,L}(\tau)$, $\mu(\Gamma_R(\tau,\epsilon,K,N))\geq 1-\eta.$ Note that we have : $$p\mu(p\Gamma_R(\tau,\epsilon,K,N))=\mu(p\Gamma_R(\tau,\epsilon,K,N)\times (H_R^N)^{t})\geq\mu(\Gamma_R(\tau,\epsilon,K,N))\geq 1-\eta.$$ 
  Thus, as in proposition \ref{main}, we get $  \tchi(\mathbf{X_1},...,\mathbf{X_n}:\mathbf{Y})\leq (1-\eta)\chi(\mathbf{X_1},...,\mathbf{X_n}:\mathbf{Y}).$
} 
 
 {Consider now the case without $\mathbf{Y}$, the only remaining inequality is $\chi(\mathbf{X_1};...;\mathbf{X_n})\leq \tchi_I(\mathbf{X_1},...,\mathbf{X_n})$.} First, note that~:
$$\text{Ent}(C_{N,\epsilon,K})=\text{Ent}(C_{N,\epsilon,K}|S_{N,\epsilon_2,K_2})+\text{Ent}(S_{N,\epsilon_2,K_2}).$$ 
{
Indeed by its definition as I-projection of the measure $S_{N,\epsilon_2,K_2}$, we know that
 $C_{N,\epsilon,K}$ has a density with respect to $S_{N,\epsilon_2,K_2}$,} and since $S_{N,\epsilon_2,K_2}$ is Lebesgue measure normalized on some set, the density with respect to Lebesgue measure does not change except for a constant and the equality above is thus easy.
 
% {If $\mathbf{Y}=\emptyset$, for any $\mu \in A_{N,(\epsilon_2,\epsilon_2),(K_2,K_2)}(\tau)$ we have by definition of Csizar projection : $$Ent(C_{N,\epsilon,K}|S_{N,\epsilon_2,K_2})\geq Ent(\mu|S_{N,\epsilon_2,K_2})$$
% }
  We can also consider $R_{N,\epsilon_2,K_2}$ the normalized Lebesgue measure on $\Gamma_R(\tau,\epsilon_2,K_2,N)$ so that~:
\begin{align*}\log(\mathrm{Leb}(\Gamma_R(\tau,\epsilon_2,K_2,N)))&=Ent(R_{N,\epsilon_2,K_2})=Ent(R_{N,\epsilon_2,K_2}|S_{N,\epsilon_2,K_2})+Ent(S_{N,\epsilon_2,K_2}),\end{align*}
 the last equality coming from inclusion of the support of $R$ in the support of $S$, both being normalized Lebesgue measure on subsets. Finally, by definition of I-projection, we get the  inequality :
$$Ent(C_{N,\epsilon,K}|S_{N,\epsilon_2,K_2})\geq Ent(R_{N,\epsilon_2,K_2}|S_{N,\epsilon_2,K_2}).$$

{ As a consequence, we also get :
$$\frac{1}{N^2}\log(\mathrm{Leb}(\Gamma_R(\tau,\epsilon_2,K_2,N)))+\frac{n}{2}\log N\leq \frac{1}{N^2}Ent(C_{N,\epsilon,K})+\frac{n}{2}\log N,$$
and we can take successively limits in $N,K_2,\epsilon_2,K_1,\epsilon_1,R$ to conclude. 

Note that it is not obvious that in general we could have $\chi(\mathbf{X_1};...;\mathbf{X_n}:\mathbf{Y})\leq \tchi_I(\mathbf{X_1},...,\mathbf{X_n}:\mathbf{Y})$.\begin{comment} without supplementary assumptions on $\mathbf{Y}$, it remains to prove this result under the stated extra assumption using our next theorem \ref{Imut}. From (3) there, we get :\begin{align*}\totchi_{I}(\mathbf{X_1}:\mathbf{Y})&=\totchi_{I}(\mathbf{X_1},\mathbf{Y})- \totchi_{Imut}(\mathbf{X_1};\mathbf{Y})-\totchi_{I}(\mathbf{Y}:\mathbf{X_{1}})
\\ &\geq \tochi(\mathbf{X_1},\mathbf{Y})-\tochi_{Mmut}(\mathbf{X_1};\mathbf{Y})-\tochi(\mathbf{Y}:\mathbf{X_{1}})
\end{align*}

We used in the second inequality $\tochi(\mathbf{X_1},\mathbf{Y})=\totchi_{I}(\mathbf{X_1},\mathbf{Y})$ we just proved in the extremal case without variables in presence, the inequality in proposition \ref{rel} and another inequality we justed proved :
$\totchi_{I}(\mathbf{Y}:\mathbf{X_{1}})\leq \tochi(\mathbf{Y}:\mathbf{X_{1}})=\chi^{\omega}(\mathbf{Y}:\mathbf{X_{1}})$ in the extremal case. In the last line we used the ultrafilter variant of Theorem \ref{Mmut} (3)\end{comment}
}
\end{proof}

  \begin{theorem}\label{Imut}

\begin{enumerate}
\item (Vanishing for one variable) 
$$\tachi_{Imut}(\mathbf{X_1})=0,$${\bl for $\mathbf{X_1}$ having finite-dimensional approximants.}

\item (Improved Subadditivity) \begin{align*}&\tchi_{Imut}(\mathbf{X_1};...;\mathbf{X_m};\mathbf{X_{m+1}};...;\mathbf{X_{n}}:\mathbf{Y})\leq{\bl \tchi_{Imut}(\mathbf{X_1}\cup...\cup\mathbf{X_m};\mathbf{X_{m+1}}\cup...\cup\mathbf{X_{n}}:\mathbf{Y})+}\\ & \tchi_{Imut}(\mathbf{X_1};...;\mathbf{X_m}:\mathbf{X_{m+1}}\cup...
\cup\mathbf{X_{n}}\cup\mathbf{Y})+\tchi_{Imut}(\mathbf{X_{m+1}};...;\mathbf{X_n}:\mathbf{X_{1}}\cup...\cup\mathbf{X_{m}}\cup\mathbf{Y}),
\\
\\ 
&\totchi_{Imut}(\mathbf{X_1};...;\mathbf{X_m};\mathbf{X_{m+1}};...;\mathbf{X_{n}}:\mathbf{Y})={\bl \totchi_{Imut}(\mathbf{X_1}\cup...\cup\mathbf{X_m};\mathbf{X_{m+1}}\cup...\cup\mathbf{X_{n}}:\mathbf{Y})+}\\& \totchi_{Imut}(\mathbf{X_1};...;\mathbf{X_m}:\mathbf{X_{m+1}}\cup...\cup\mathbf{X_{n}}\cup\mathbf{Y})+\totchi_{Imut}(\mathbf{X_{m+1}};...;\mathbf{X_n}:\mathbf{X_{1}}\cup...\cup\mathbf{X_{m}}\cup\mathbf{Y}).
\end{align*}

\item (Improved subadditivity of entropy) \begin{align*}\tchi_{I}&(\mathbf{X_1},\mathbf{X_{2}}:\mathbf{Y})\leq \tchi_{Imut}(\mathbf{X_1};\mathbf{X_{2}}:\mathbf{Y})+ \tchi_{I}(\mathbf{X_1}:\mathbf{X_{2}}\cup\mathbf{Y})+\tchi_{I}(\mathbf{X_2}:\mathbf{X_{1}}\cup\mathbf{Y}),
\\
\\
\totchi_{I}&(\mathbf{X_1},\mathbf{X_{2}}:\mathbf{Y})= \totchi_{Imut}(\mathbf{X_1};\mathbf{X_{2}}:\mathbf{Y})+\totchi_{I}(\mathbf{X_1}:\mathbf{X_{2}}\cup\mathbf{Y})+\totchi_{I}(\mathbf{X_2}:\mathbf{X_{1}}\cup\mathbf{Y}).
 \end{align*}
 \item (Agreement with previous definition) 

 If $\mathbf{X_i}$ are hyperfinite multivariables {(more accurately $\mathbf{P_i}=1$)} then
$$\chi_{orb}(\mathbf{X_{1}};...;\mathbf{X_n})\leq \tchi_{Imut}(\mathbf{X_{1}};...;\mathbf{X_n}).$$
({$\chi_{orb}$} in the sense of \cite{HiaiMUeda}).
If  moreover $\tau_{\mathbf{X_{1}},...,\mathbf{X_{n}}}$ is extremal then
 $$\tchi_{Imut}(\mathbf{X_{1}};...;\mathbf{X_n})=\chi_{orb}(\mathbf{X_{1}};...;\mathbf{X_n}).$$ 
\end{enumerate}
 \end{theorem}
 
\begin{proof}
{\sl $(1)$} Similar to $\tchi_{orb}$.

\medskip
{\sl $(2),(3)$} These follow from equalities in the corresponding proofs for  $\tchi_{Mmut}$.

{\sl $(4)$} After using Theorem \ref{prop}(6) in case of extremality and relating inequalities of our variants (proposition \ref{rel}), it remains to prove : $\chi_{orb}(\mathbf{X_{1}};...;\mathbf{X_n})\leq \tchi_{Imut}(\mathbf{X_{1}};...;\mathbf{X_n})$.

We take notations of \cite{HiaiMUeda} especially $\Xi_i(N)$ (as in lemma 4.2 and definition 4.1 there) is a sequence approximating the hyperfinite variables $\mathbf{X_i}$ in mixed moments. 
 We now show that, for every $\epsilon_1,K_1$, there exists $\delta,L$ such that, for every 
  $\epsilon=(\epsilon_1,\epsilon_2),\epsilon_2\leq\delta,\ K=(K_1,K_2),K_2\geq L$~:
\begin{align*}&\limsup_{N\to \infty} \frac{1}{N^2}\log \gamma_{N,\Xi(N),\epsilon_1/2,K}\leq \limsup_{N\to \infty} \frac{1}{N^2} \text{Ent}( C_{N,\epsilon,K}|D_{N,\epsilon,K}),
\end{align*}
where $C_{N,\epsilon,K}=C_{N,\epsilon,K}(\tau_{\mathbf{X_1},...,\mathbf{X_n}})$, $D_{N,\epsilon,K}=p_1UpC_{N,\epsilon,K}\otimes... \otimes p_n UpC_{N,\epsilon,K}=p_1C_{N,\epsilon,K}\otimes... \otimes p_n C_{N,\epsilon,K}$.
First, we use  Jung's Lemma  and follow the proof of Lemma 4.2 in \cite{HiaiMUeda}. We can thus take $\delta,L$ such that,
for all families of sets  $(\Theta_i)_{i=1,\ldots,n}$ of $N\times N$ hermitian matrices, for $N$ large enough, with  $\tau_{(\Theta_i)}\in V_{\delta,L}(p_i\tau)$ for all $i$,
 {\gree we have}~:
\begin{equation}\label{thetaxi2}
\gamma_{N,\Theta,\epsilon_1,K_1}{\bl \geq }\gamma_{N,\Xi(N),\epsilon_1/2,K_1}.
\end{equation}
 Moreover, using again lemma \ref{UEnt} (ii), \begin{align*}\text{Ent}(C_{N,\epsilon,K}|D_{N,\epsilon,K})&=\text{Ent}(C_{N,\epsilon,K}|S_{N,\epsilon_2,K_2}(\tau))-\text{Ent}(D_{N,\epsilon,K}|S_{N,\epsilon_2,K_2}(\tau))\\ &\geq \text{Ent}(C_{N,\epsilon,K}|S_{N,\epsilon_2,K_2}(\tau)).\end{align*}
  
  In order to use the definition of Csizar's projection, we have to take a specific measure in  $A_{N,\epsilon,K}$. Note that  we have considered Csizar's projection with respect to $S_{N,\epsilon_2,K_2}(\tau)$,  in order to have a measure with support included in a set where hyperfinite variables for marginals will be of the form $\Xi'$, for which we can apply the relation (\ref{thetaxi2}) above.
Let $$dT_{N,\epsilon,K}(\Xi')=\frac{1_{\Xi'\in \Gamma_R(\mathbf{X_1},...,\mathbf{X_n},N,K,\epsilon_1)}}{\gamma_{N,\Xi',\epsilon_1,K_1}}  d(S_{N,\epsilon_2,K_2}(\tau))(\Xi').$$
 This is  a probability measure: since $S_{N,\epsilon_2,K_2}$ is an $U(N)^n$ invariant probability we can  compute the total mass by integrating the density over unitaries and by definition 
\begin{align*}{\mathcal H}_N^n(1_{U\Xi'U^*\in \Gamma_R(\mathbf{X},N,K_1,\epsilon_1)})&={\mathcal H}_N^n(\Gamma_{orb}(\mathbf{X_1},...,\mathbf{X_n}:\Xi_1',...\Xi_n',N,K_1,\epsilon_1))=\gamma_{N,\Xi',\epsilon_1,K_1}.
\end{align*}
 From this and since its support is in $\Gamma_R(\mathbf{X_1},...,\mathbf{X_n},N,K_1,\epsilon_1)$ we deduce that $T_{N,\epsilon,K}\in A_{N,\epsilon,K}$.
  
It follows, by definition of $C$ as Csiszar's projection of $S$, that  \begin{align*}\text{Ent}(C_{N,\epsilon,K}|S_{N,\epsilon_2,K_2})&\geq \text{Ent}(T_{N,\epsilon,K}|S_{N,\epsilon_2,K_2})=T_{N,\epsilon,K}(\log (\gamma_{N,.,\epsilon_1,K_1}))\\&\geq T_{N,\epsilon,K}(\log (\gamma_{N,\Xi(N),\epsilon_1/2,K_1}))=\log (\gamma_{N,\Xi(N),\epsilon_1/2,K_1}).\end{align*}
  The second inequality comes from (\ref{thetaxi2}) since $\epsilon_2\leq \delta$, $K_2\geq L$.
This concludes. \end{proof} 
\subsection{Comparison of the various entropies}
Beyond the case of equality in the context of \cite{HiaiMUeda}, we have the following general inequality.
\begin{proposition}\label{rel}(Relating Inequalities) $$\tchi_{Imut}(\mathbf{X_1};...;\mathbf{X_n}:\mathbf{Y})\leq\tchi_{Mmut}(\mathbf{X_1};...;\mathbf{X_n}:\mathbf{Y})\leq \tchi_{orb}(\mathbf{X_1};...;\mathbf{X_n}:\mathbf{Y})\leq0.$$ 
\end{proposition}
\begin{proof}
Negativity comes from negativity of relative entropy. The first inequality follows from  $C_{N,\epsilon,K}(\tau)\in A_{N,\epsilon,K}(\tau)\subset A_{N,\epsilon_1,K_1}(\tau)$ (for $K_1\leq K_2, \epsilon_1\geq\epsilon_2$)  and our conventions in case this is empty.

Finally, applying lemma \ref{UEnt} (i) for any $\mu\in A_{N,\epsilon,K}(\tau)$ {\gree we get} the inequality~:
$$Ent(p\mu|p_1Up\mu\otimes\ldots\otimes p_nUp\mu)\leq Ent(p\mu|Up\mu).$$
The second inequality follows.
 \end{proof}
\bibliographystyle{amsalpha}

\end{document}